\documentclass[11pt]{article}

\usepackage[T1]{fontenc}
\usepackage{lmodern}
\usepackage{amsmath,amssymb,amsthm,mathtools}
\usepackage[margin=1in]{geometry}
\usepackage{microtype}
\usepackage{enumitem}
\usepackage{xcolor}
\usepackage[colorlinks=true,linkcolor=blue!55!black,citecolor=blue!55!black,urlcolor=blue!55!black]{hyperref}
\usepackage{todonotes,hyperref}
\usepackage{listings}
\usetikzlibrary{arrows.meta}
\usepackage{authblk}
\usepackage{tabularray}
\newcommand{\email}[1]{\texttt{#1}}
\theoremstyle{definition}
\newtheorem{theorem}{Theorem}[section]
\newtheorem{lemma}[theorem]{Lemma}
\newtheorem{proposition}[theorem]{Proposition}
\newtheorem{corollary}[theorem]{Corollary}
\newtheorem{conjecture}[theorem]{Conjecture}
\newtheorem{definition}[theorem]{Definition}
\theoremstyle{remark} % 
\newtheorem{remark}[theorem]{Remark}

\usepackage{cleveref}
\usepackage{appendix}

\newcommand{\Cov}{\operatorname{Cov}}
\newcommand{\Var}{\operatorname{Var}}

\newcommand{\avdeg}{\bar d}
\newcommand{\E}{\mathbb{E}}
\newcommand{\Pp}{\mathbb{P}}
\newcommand{\R}{\mathbb{R}}
\newcommand{\one}{\mathbf{1}}
\newcommand{\tr}{\operatorname{tr}}
\newcommand{\diag}{\operatorname{diag}}

\newcommand{\K}{{\kappa}}
\newcommand{\opnorm}[1]{\lVert #1\rVert_{\mathrm{op}}}
\newcommand{\nucnorm}[1]{\lVert #1\rVert_{*}}

\newcommand{\Gnp}{G(n,p)}

\DeclareMathOperator{\spec}{\operatorname{spec}}

\title{The Kemeny--Braess Paradox in Random Graphs}

\author{%
${}$\\
  Manuel Bodirsky\textsuperscript{\textdagger}\qquad
  Lorenzo Ciardo\textsuperscript{\ddag}\qquad
  Mihyun Kang\textsuperscript{\S}
  \\${}$\\
  \begin{tabular}{@{}c@{}}
    \small\textsuperscript{\textdagger}Chair of Algebra and Discrete Structures,
    TU Dresden\\
    \footnotesize\email{manuel.bodirsky@tu-dresden.de}\\${}$\\
    \small\textsuperscript{\ddag}Institute of Algorithms and Theory,
    Graz University of Technology\\
    \footnotesize\email{lorenzo.ciardo@tugraz.at}\\${}$\\
    \small\textsuperscript{\S}Institute of Discrete Mathematics,
    Graz University of Technology\\
    \footnotesize\email{kang@math.tugraz.at}
  \end{tabular}%
}

\date{\today}

\begin{document}

\maketitle

\begin{abstract}
In 2016, Kirkland and Zeng initiated the mathematical treatment of Braess paradox
for random walks on graphs using Kemeny's constant as the measure
of travel time. 
We study this phenomenon for the Erd\H{o}s--R\'enyi
random graph $\Gnp$.
An edge in a connected graph  is called {\em Kemeny--Braess} if its deletion preserves connectedness and
decreases Kemeny's constant, while a nonedge is {\em Kemeny--Braess} if its
insertion increases Kemeny's constant. Let $\avdeg=(n-1)p$ be the expected degree of $G\sim \Gnp$, and assume
that $\avdeg-\log n\to\infty$, so that $G$ is connected with high
probability.

We determine the sharp asymptotic thresholds for the existence of
both types of paradoxical behaviour. With
$C^*=1/(1-\log 2)$, the threshold for Kemeny--Braess edges is
$\avdeg=C^*\log n$, whereas the threshold for Kemeny--Braess nonedges is
$\avdeg=2C^*\log n$. Hence, as the expected degree $\avdeg$ increases, the connected
regime exhibits three phases: first Kemeny--Braess paradox occurs for both edges and nonedges, then only
for nonedges, and finally it disappears. In particular, in the dense regime, with high probability  
$G$ has no Kemeny--Braess edges or nonedges. This contrasts sharply with the spectral gap of the normalised
Laplacian, for which Braess' paradox persists at constant edge density.
%, 

Our proof combines a spectral-moment
expansion of Kemeny's constant with uniform edge-perturbation estimates. Near
the thresholds, the sign of the perturbation is governed by the endpoint
degrees, and binomial large deviations yield the constant $C^*$ and the
factor of two between the two thresholds.
\end{abstract}

\clearpage

\newpage

\tableofcontents

\section{Introduction}
A counterintuitive phenomenon in the theory of congestion games, first observed by Braess in 1968, is that enlarging a traffic network may degrade its performance: adding a new road to a network in which drivers route themselves selfishly can strictly increase the travel time experienced by every driver at equilibrium, as well as the time it takes for traffic to spread evenly through the network. While this phenomenon---now known as the \textit{Braess paradox}---was originally observed in a game-theoretic context, where it emerges from the misalignment between selfish behaviour and social optimality, it was later recognised to occur in settings entirely devoid of strategic agents. In particular, the paradox manifests itself as a purely graph-theoretic phenomenon in the limit behaviour of random walks: Inserting an edge into a graph may slow down the random walk on it, in the sense that natural quantities measuring the expected first-passage time and the time to mixing may strictly increase. In this form, the paradox is intrinsic to the combinatorial structure of the graph, and its occurrence becomes a question about which graphs, and which edges, exhibit it.

A fully graph-theoretic formalisation of the phenomenon requires a quantitative notion of the quality of the traffic flow in a network as a parameter of the corresponding connected graph. The theory of Markov chains provides  a classic invariant measuring the expected travel time in a network, known as \textit{Kemeny's constant}~\cite{KemenySnell}.
If an irreducible Markov chain on state space $\Omega$ has stationary distribution $\pi$,
write
$H(x,y)$ for the \emph{expected hitting time} of $y$ from $x$, that is,
the expected number of steps after which the chain started at $x$ first
visits $y$, with the convention $H(x,x)=0$.  Then the quantity
\[
  \K := \sum_{y\in \Omega}\pi(y)H(x,y)
\]
does not depend on the starting state $x$; 
this observation goes back to
Kemeny and Snell \cite{KemenySnell} and is commonly called the \emph{random
target lemma} \cite[Lemma~10.1]{LevinPeresWilmer}.
The common value---Kemeny's constant---is
thus an intrinsic parameter of the chain: it is the expected time to reach a
target drawn from the stationary distribution, and it may be read as a single
scalar summary of how quickly the chain travels between typical states.
Note that if the chain is a simple random walk on a connected graph $G$, then $\K$ becomes a
graph invariant.

Beyond its role in Markov-chain theory---see the survey by
Hunter \cite{Hunter}---Kemeny's constant has appeared in applications ranging
from web navigation \cite{LeveneLoizou}, to vehicular traffic on road networks
\cite{CrisostomiKirklandShorten}, where a small value of $\K$ indicates a
well-connected system in which typical destinations are reached quickly, to computational chemistry, where it has proved to
help characterise molecular conformation dynamics~\cite{martino2024kemeny}, to epidemiology, where it
has been used to measure the spread rate of infectious diseases, predicting how quickly
an outbreak reaches epidemic levels~\cite{yilmaz2020kemeny,dudkina2024comparison}.

Kemeny's constant is intimately related to
the \textit{spectral gap} of the
normalised Laplacian, another classic parameter in the study of random
walks, closely linked to relaxation and mixing properties of the Markov chain~\cite{montenegro2006mathematical,Chung97}. 
For %
a graph $G$, this is defined as the second-smallest normalised Laplacian eigenvalue of $G$. %
Given
a {connected} graph with transition-matrix eigenvalues
$1=\lambda_1 > \lambda_2\geq\ldots\geq\lambda_n$,
Kemeny's constant admits the spectral representation 
\begin{align}
\label{eqn_spectral_formula_kemeny_constant}
\K =\K(G) =\sum_{i=2}^{n}\frac{1}{1-\lambda_i}
\end{align}
see, e.g.,~\cite[\S3]{lovasz1993random} and~\cite[\S3]{LeveneLoizou}. 
Since the random-walk transition matrix and the normalised adjacency matrix have the same spectrum, the normalised Laplacian eigenvalues are $1-\lambda_i$. Consequently, the reciprocal of the spectral gap $1-\lambda_2$ is the largest summand in the spectral formula~\eqref{eqn_spectral_formula_kemeny_constant} for Kemeny’s constant. 

\paragraph*{Braess edges and nonedges.}
The study of Braess paradox for
Kemeny's constant was pioneered by Kirkland and Zeng
in~\cite{KirklandZeng}.
Let $G$ be an arbitrary connected graph.
We say that a nonedge $uv$ of $G$ is a
\emph{Kemeny--Braess nonedge} (in short, \emph{$\kappa$-Braess nonedge}) 
if
\[
  \K(G+uv)>\K(G).
\]
 Moreover, an edge $e$ of $G$ is a
 \emph{Kemeny--Braess edge}
(in short, \emph{$\kappa$-Braess edge}) %
if $G-e$ is connected and
\[
  \K(G-e)<\K(G).
\]
Following~\cite{KirklandZeng}, several papers identified structural criteria
for the occurrence of the paradox in various
classes of graphs; see the discussion in the ``Related work'' paragraph.

Analogously, we call an edge in a connected graph 
\textit{$\operatorname{sg}$-Braess} if deleting it preserves connectivity and
increases the spectral gap, and call a nonedge \textit{$\operatorname{sg}$-Braess}
if adding it decreases the spectral gap.
Note that all four notions ($\K$-Braess edges/nonedges,
$\operatorname{sg}$-Braess edges/nonedges)
express that travelling in the graph is harder when the
extra road is available, and are thus all instances of Braess paradox for graphs.
It was conjectured by Fan Chung in 2014 that almost all dense Erd\H{o}s--R\'enyi random graphs have many  sg-Braess edges:

\begin{conjecture}[Chung, 2014; see {\cite[Conjecture~1]{eldan2017braess}}]
\label{conj_chung}
{For each fixed $p\in(0,1)$, there exists a constant $c=c(p)>0$
such that
at least $c$ fraction
of the edges of $G(n,p)$ are $sg$-Braess whp.\footnote{\emph{With high probability}, meaning with probability tending to one as $n$ approaches $\infty$.}}
\end{conjecture}
Eldan, R\'acz, and Schramm proved the corresponding result for sg-Braess nonedges in 2017.

 \begin{theorem}[\cite{eldan2017braess}]
 \label{thm_eldan}
     For each fixed $p\in(0,1)$, there exists a constant $c=c(p)>0$ such that at least $c$ fraction of the nonedges in $G(n,p)$ are $sg$-Braess whp.
 \end{theorem}

 \noindent In addition, the same work proved a weaker version of Chung's Conjecture~\ref{conj_chung}.

 \begin{theorem}[\protect{\cite{eldan2017braess}}]
\label{thm:eldan-edges}
For every fixed $p\in(0,1)$, there exists a constant $c=c(p)>0$
such that, for $G\sim G(n,p)$,
\[
  \limsup_{n\to\infty}
  \Pp\left(
    \bigl|\{e\in E(G): e\text{ is }sg\text{-Braess}\}\bigr|
    \ge c|E(G)|
  \right)
  \ge c.
\]
\end{theorem}

Hence, dense random graphs appear to exhibit \textit{many} instances of Braess paradox when the network performance is measured according to the spectral gap.
It is then natural to investigate whether the analogues of Conjecture~\ref{conj_chung} and~\Cref{thm_eldan} hold for Kemeny's constant.

\paragraph*{Contributions.}
Our two main results determine, throughout the connected Erd\H{o}s--R\'enyi regime, where
Kemeny--Braess edges and nonedges disappear. 
Define the constant\footnote{Throughout the paper, $\log$ denotes the natural logarithm.}
\[C^*=\frac{1}{1-\log 2}=3.25889\ldots\]

\begin{theorem}[Threshold for Kemeny--Braess edges]
\label{thm:edge-threshold}
Fix $\varepsilon>0$. For $p=p(n)\in[0,1]$, let $\avdeg := (n-1) p$ denote the expected degree of a vertex in $G(n,p)$, and suppose that  $\avdeg-\log n\to\infty$.
Then $G\sim G(n,p)$ is connected whp, and the following holds.
\begin{enumerate}%
  \item[(i)] If $\avdeg \le (C^*-\varepsilon)\log n$ for all sufficiently large $n$,
  then $G$ has a $\kappa$-Braess edge whp.
  \item[(ii)] If $\avdeg \ge (C^*+\varepsilon)\log n$ for all sufficiently large $n$,
  then $G$ does not have any $\K$-Braess edge whp.\footnote{In particular, whp, $G-e$ is connected and $\K(G-e)>\K(G)$
  for every edge $e$ of $G$.}
\end{enumerate}
\end{theorem}

\begin{theorem}[Threshold for Kemeny--Braess nonedges]
\label{thm:nonedge-threshold}
Fix $\varepsilon>0$. For $p=p(n)\in[0,1]$, let $\avdeg := (n-1) p$ denote the expected degree of a vertex in $G(n,p)$, and suppose that  $\avdeg-\log n\to\infty$.
Then $G\sim G(n,p)$ is connected whp, and the following holds.
\begin{enumerate}%
  \item[(i)] If $\avdeg \le (2C^*-\varepsilon)\log n$ for all sufficiently large $n$,
  then $G$ has a $\kappa$-Braess nonedge whp.
  \item[(ii)] If $\avdeg \ge (2C^*+\varepsilon)\log n$ for all sufficiently large $n$,
  then $G$ does not have any $\K$-Braess nonedge whp.%
\end{enumerate}
\end{theorem}

Together with the classical connectivity threshold at
  $\avdeg\sim\log n$~\cite{ErdosRenyi}, our results give
the phase diagram shown in~\Cref{fig:braess-phases}. %
In particular,
for every fixed $p\in(0,1]$, the upper
halves of the two theorems show that, whp, there is no Kemeny--Braess 
paradoxical edge or nonedge at all.  Thus, the Kemeny analogues of both {Conjecture}~\ref{conj_chung} and \Cref{thm_eldan} fail in the strongest possible sense.
Hence, despite the close spectral relationship between Kemeny's constant and the
spectral gap, the two parameters behave very differently on dense
random graphs.

\begin{figure}%
\centering
\definecolor{phasegrey}{HTML}{737980}
\definecolor{phaseorange}{HTML}{A95D31}
\definecolor{phaseblue}{HTML}{326F91}
\definecolor{phasegreen}{HTML}{3D7563}
\begin{tikzpicture}[
  x={\dimexpr(\linewidth-2pt)/160\relax}, y=1cm,
  phase/.style={align=center, font=\small\bfseries,
                text width=.24\linewidth, inner sep=0pt},
  source/.style={align=center, font=\footnotesize,
                 text width=.24\linewidth, inner sep=0pt,
                 text=black!65}
]
  \draw[black!30, line width=.45pt]
    (30,1.43) -- (30,1.58) -- (158,1.58) -- (158,1.43);
  \node[fill=white, inner xsep=7pt, font=\small] at (94,1.58)
    {Connected regime: $\kappa$-Braess paradox};

  \node[phase, text=phasegrey, text width=.18\linewidth]
    at (15,.88) {Disconnected};
  \node[phase, text=phaseorange] at (51.5,.88)
    {Edges and nonedges};
  \node[phase, text=phaseblue] at (94.5,.88)
    {Nonedges only};
  \node[phase, text=phasegreen] at (137,.96) {Neither};
  \node[font=\footnotesize, text=phasegreen] at (137,.57)
    {%
    };

  \fill[phasegrey!22]   (0,.12)    rectangle (29.5,.32);
  \fill[phaseorange!25] (30.5,.12) rectangle (72.5,.32);
  \fill[phaseblue!25]   (73.5,.12) rectangle (115.5,.32);
  \fill[phasegreen!25]  (116.5,.12) rectangle (157.5,.32);
  \draw[-{Stealth[length=2.4mm,width=1.7mm]}, line width=.65pt]
    (0,0) -- (160,0);

  \foreach \x/\lab in {30/1,73/C^*,116/2C^*}{
    \draw[line width=.55pt] (\x,-.08) -- (\x,.39);
    \node[anchor=north, inner sep=2pt, font=\small]
      at (\x,-.12) {$\lab$};
  }
  \node[anchor=north east, inner sep=2pt, font=\small]
    at (160,-.12) {$\avdeg/\log n$};

  \node[source, text width=.185\linewidth] at (15,-.88)
    {};
  \node[source] at (51.5,-.88)
    {Thm.~\ref{thm:edge-threshold}$(i)$\\
     Thm.~\ref{thm:nonedge-threshold}$(i)$};
  \node[source] at (94.5,-.88)
    {Thm.~\ref{thm:edge-threshold}$(ii)$\\
     Thm.~\ref{thm:nonedge-threshold}$(i)$};
  \node[source] at (137,-.88)
    {Thm.~\ref{thm:edge-threshold}$(ii)$\\
     Thm.~\ref{thm:nonedge-threshold}$(ii)$};
\end{tikzpicture}
\caption{Phase diagram of the occurrence of the $\kappa$-Braess paradox in the Erd\H{o}s--R\'enyi
random graph $G(n,p)$, where
  $\avdeg=(n-1)p$ and $C^*=1/(1-\log 2)$.
  All conclusions hold whp 
  outside
 fixed-width multiplicative neighbourhoods of the thresholds $C^*$ and $2C^*$ on the
  $\avdeg/\log n$ axis.}
\label{fig:braess-phases}
\end{figure}
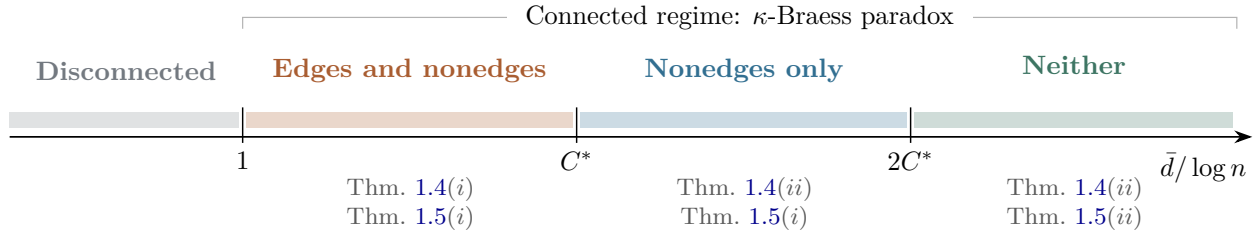

\paragraph*{Proof strategy}
First of all, in~\Cref{sec_spectral_expansion}, 
we split Kemeny's constant into a sum of
spectral moments of the matrix %
obtained from the normalised adjacency matrix by removing the principal eigenspace, 
{following the standard spectral decomposition method for hitting times from~\cite[\S3]{lovasz1993random}.}
We then treat
the dense regime first. Throughout~\Cref{sec:dense-uniform-estimates},
we allow $p=p(n)$ to vary while remaining bounded away from zero, compute
the quadratic and cubic spectral moments explicitly, and control all
higher moments by operator-norm estimates. This yields the stronger
conclusion that, whp, adding any nonedge decreases
Kemeny's constant and deleting any edge increases it.

The sparse regime $p=o(1)$ is treated in
\Cref{sec:logarithmic-sparse-regime}. There we derive an
edge-perturbation formula and combine it with uniform random-graph
estimates, existence results for suitable low-degree vertices and
edges, and large-deviation bounds. The
lower-threshold statements automatically belong to this sparse regime.
For the upper-threshold statements, we combine the sparse and dense
arguments by passing to subsequences: along any subsequence, either
$p\to0$, in which case the sparse argument applies, or $p$ is bounded
away from zero on a further subsequence, in which case the dense argument
applies. This proves the upper-threshold statements for arbitrary
$p=p(n)$.

The distinction between the constants has a simple probabilistic origin.
For both additions and deletions, the first-order sign is governed by whether
the sum of the endpoint degrees is smaller or larger than $\avdeg$.  The
large-deviation cost of two endpoints near degree $\avdeg/2$ is
$1-\log2$.  There are order $n^2$ candidate nonedges but only
order $n^2p$,  {equivalently order $\avdeg n$}, candidate edges, which doubles the nonedge
threshold constant.

\paragraph*{Related work}
The programme of understanding network performance via the analysis of Kemeny's constant and the identification of Braess-generating structures in graphs has recently attracted a considerable and growing amount of attention across research communities in matrix theory, random graph theory, optimisation, and theoretical computer science. Much of the work in this area concentrates on estimating Kemeny's constant value for classes of graphs or natural regimes of random graphs through
spectral methods, perturbation theory, and
combinatorial formulae, often based on the effective resistance and Kirchhoff index.
Spectral approaches include approximation by sparsification~\cite{abiad2025spectral}
and bounds obtained through matrix compression, pinching inequalities, and
eigenvalue interlacing~\cite{abiad2026kemenysconstantmatrixcompression}. 
For Erd\H{o}s--R\'enyi graphs, asymptotic and concentration estimates for
Kemeny's constant have been established in various density regimes
in~\cite{lowe2014hitting,sylvester2021random,lowe2023central} via spectral methods.
Combinatorial formulae and bounds on Kemeny's constant have been obtained for
graphs assembled from highly symmetric components~\cite{palacios2020kemeny,palacios2021kemeny}, flower graphs~\cite{faught2020resistance},
trees of prescribed order and diameter~\cite{ciardo2022kemeny}, and threshold graphs~\cite{breen2025threshold}.
Various decomposition methods for computing Kemeny's constant have been proposed, via techniques exploiting harmonic eigenvectors for twin subgraphs~\cite{breen2019computing},
stochastic complements~\cite{bini2024kemeny},
cut vertices~\cite{faught20221}, and bridges~\cite{breen2022kemeny}.
Further work addresses sensitivity to changes in transition probabilities~\cite{breen2019structured}, minimisation under structural constraints for reversible Markov chains~\cite{durastante2025kemeny}, analogues for non-backtracking walks~\cite{breen2026defining}, as well as diameter bounds and related Nordhaus--Gaddum problems for Kemeny's constant~\cite{li2026conjecturesconcerningkemenysconstant,knudson2024nordhaus,kim2026bounds}.
Combinatorial formulae based on the resistance distance~\cite{KleinRandic1993,koolen2013electric,klein2002resistance,bapat1999resistance,wang2017kemeny} and the 2-forest matrix and Kirchhoff's theorem~\cite{KirklandZeng,barrett2020spanning} have proved effective at linking the occurrence of Braess paradox for Kemeny's constant to constant-size pendent substructures (``limbs'') of the base graph. In particular, it was shown in~\cite{KirklandZeng} that a cherry pair induces a $\K$-Braess nonedge in any tree on at least four vertices.\footnote{For example,
the three-leaf star $K_{1,3}$ satisfies $\K=5/2$, whereas joining two leaves
gives $\K=61/24>5/2$ {(this follows e.g.,~from~\cite[Remark 3.2]{KirklandZeng})}.}  This was used in the same paper to obtain a Schwenk-type result on the occurrence of $\K$-Braess nonedges in almost all trees.
Later work 
showed that cherry pairs induce $\K$-Braess nonedges in any connected graph on at least $4$ vertices~\cite{ciardo2020braess}. Combining this with asymptotic results from probabilistic combinatorics on the occurrence of limbs in addable graph classes~\cite{mcdiarmid2005random,mcdiarmid2006random}, it was deduced therein that $\K$-Braess nonedges appear in almost all connected planar graphs.  
The analysis was further extended in~\cite{kim2021families} to the case of twin pendent paths, see also the thesis~\cite{kim2021qualitative}. Furthermore,~\cite{faught20221} used their $1$-separation formula to identify pendent configurations yielding \textit{$\K$-Braess sets}: sets of nonedges whose simultaneous addition increases Kemeny’s constant. 
Recently,~\cite{breen2026kemeny} initiated the study of $\K$-Braess sets forming a clique;\footnote{Note that $\K$-Braess nonedges are the specialisation of this notion to the case of a $2$-clique.} via tools based on the accessibility index and equitable partitions, they showed that, for any fixed $\ell$, almost all connected planar graphs have $\K$-Braess $\ell$-cliques, thus extending the results of~\cite{KirklandZeng,ciardo2020braess}.
We observe that criteria on the occurrence of $\K$-Braess paradox based on fixed limbs, such as those discussed above, cannot establish the paradox in the regime considered here, as any fixed nontrivial limb disappears whp strictly below the connectivity threshold.
As a consequence, for the nonedges constructed in the proof of~\Cref{thm:nonedge-threshold}(i), we establish the $\K$-Braess property  
via a global spectral analysis of the surrounding random graph $G(n,p)$, rather than by exhibiting a local Braess-generating substructure.
Further work 
studies $\K$-Braess paradox in complete multipartite and split graphs~\cite{hu2019complete},
enumerates $\K$-Braess nonedges in paths, spiders, and brooms~\cite{jang2025kemeny}, and determines extremal changes in Kemeny's constant caused by adding an edge to a tree~\cite{kirkland2025edge}.

\paragraph*{Acknowledgements}
M. Bodirsky is funded by the European Union (project POCOCOP, ERC Synergy grant No. 101071674). 
Views and opinions expressed are however those of the author(s) only and do not necessarily reflect those of the European Union or the European Research Council Executive Agency. Neither the European Union nor the granting authority can be held responsible for them. L. Ciardo is supported in part by the Steiermärkischen Landesregierung, Referat Wissenschaft und Forschung [grant ref. no. ABT12-77337/2026-6].  M. Kang is supported in part by the Austrian Science Fund (FWF) [10.55776/F1002].

Proof search for this work was assisted by ChatGPT. The main results emerged through multiple rounds of conversations with ChatGPT 5.6 Sol, restructuring and simplifications of the resulting arguments by the authors, and further conversations informed by these revisions. In addition, proof verification was assisted by Claude 5 Fable, Claude Opus 5.5, and ChatGPT 6 Astra. The authors take full responsibility for the mathematical content of the paper.

\paragraph*{Notation}
We denote by $\one$ the all-one vector, whose size will be clear from the context.
If $A$ denotes a matrix, then we typically denote its {entries} by $a_{ij}$  or $A_{ij}$. 
We let $\spec(A)$ denote the spectrum of a square matrix $A$.  
Given a graph $G=(V(G),E(G))$ and a vertex $u\in V(G)$, the standard basis vector indexed by $u$, {which has coordinate $1$ at $u$ and $0$ at every other vertex, is denoted by $e_u$}.
{We shall often write $V$ for $V(G)$ and $E$ for $E(G)$ when the graph $G$ is clear from the context.}
We write $N_G(u)$ for the neighbourhood of $u$ in $G$, $d_G(u)$ for the degree of $u$ in $G$,  and $t_G(u)$ for the number of triangles containing $u$ in $G$. For $u,v \in V$, we write $u \sim_G v$ to denote that $u$ and $v$ are adjacent in $G$. If the underlying graph $G$ is clear from the context, we  omit the subscript and write $N(u)$, $d_u$, $t_u$, and $u \sim v$, to ease notation. We also let $\mathbf{1}_{\{i \in \{u,v\}\}}$ equal 
1 if $i \in \{u,v\}$ and 0 otherwise. %

We refer to Glossary~\ref{glo:op} for a more detailed description of the terminology and basic notions from spectral graph theory and matrix theory that are used throughout the paper. 

\section{Preliminaries: A spectral-moment expansion}
\label{sec_spectral_expansion}
We begin by recalling standard spectral facts about random walks and normalised graph matrices; see Chung \cite[\S\S1.1--1.3]{Chung97} and Levin, Peres, and Wilmer \cite[\S12.1]{LevinPeresWilmer}. We include the details needed to fix the notation and to derive the spectral-moment expansion used below.
In the following, let $G=(V,E)$ be a connected simple undirected graph.
Let $A$ be the adjacency matrix of $G$, 
and let 
 $D=\diag(d_v:v\in V)=\diag(A\one)$ 
be the diagonal degree matrix of $G$. 
The \emph{normalised adjacency matrix}
\[
  S:=D^{-1/2}AD^{-1/2}
\]
of $G$ is similar to the \textit{transition matrix} $T:=D^{-1}A$, since $D^{-1/2}SD^{1/2}=D^{-1}A=T$,
and hence $T$ and $S$ have the same trace $\tr(T)=\tr(S)$ and the same eigenvalues.  
Note that $T$ is stochastic, and hence has the eigenvalue $1$ (which is also called the \emph{trivial} eigenvalue) with right eigenvector $\mathbf 1$.
Define
\[
  \phi:=\frac{D^{1/2}\one}{\sqrt{2|E|}},
  \qquad \Pi:=\phi\phi^{\top}=\frac{1}{2|E|}D^{1/2}\one\one^\top D^{1/2},
  \qquad B:=S-\Pi.
\]

Observe that
\begin{align*}
    S\phi
    &=
    \frac{1}{\sqrt{2|E|}}D^{-1/2}AD^{-1/2}D^{1/2}\one
    =
    \frac{1}{\sqrt{2|E|}}D^{-1/2}A\one
    =
    \frac{1}{\sqrt{2|E|}}D^{-1/2}D\one\\
    &=
    \frac{1}{\sqrt{2|E|}}D^{1/2}\one
    =\phi.
\end{align*}
Since $S$ is symmetric, we deduce that \begin{align} S\Pi=\Pi=\Pi S. \label{eq:Pi}
\end{align} As usual, let
$1  = \lambda_1 > \lambda_2\geq \ldots\geq \lambda_n$ be the eigenvalues of $T$ (and of $S$).
Note that $\phi$ has unit norm, since
\begin{align*}
    \phi^\top\phi=\frac{1}{2|E|}\one^\top D\one=\frac{2|E|}{2|E|}=1.
\end{align*}
It follows that $\Pi$ is idempotent: $\Pi^2=\Pi$. 
Moreover, $B\phi=S\phi - \Pi\phi=\phi-\phi=0$, and $Bv=Sv$ for each vector $v$ orthogonal to $\phi$. 
Thus, $B$ is obtained from $S$ by removing its principal eigenspace.
It follows that $\spec(B)=(0,\lambda_2,\ldots,\lambda_n)$.
Moreover,
$\tr(S)=0$ (since $G$ has no self-loops) and $\tr(\Pi) =\phi^\top\phi= 1$, so $\tr(B)=-1$.

Next, we decompose Kemeny's constant through the standard spectral decomposition method for hitting times from~\cite[\S3]{lovasz1993random}; see also~\cite{helali2019hitting,lowe2014hitting,lowe2023central}.
Whenever $\opnorm{B}<1$ (i.e., whenever $G$ is connected and nonbipartite; see  Glossary~\ref{glo:op}), expanding the spectral formula for Kemeny's constant in~\eqref{eqn_spectral_formula_kemeny_constant} via a geometric series gives
\begin{align}
  \K(G)
    =\sum_{i=2}^n\frac{1}{1-\lambda_i} 
    & =\sum_{i=2}^n\sum_{r\geq 0}\lambda_i^r \notag && \text{(here we use $\opnorm{B} <1$)}\\
    &
    =\sum_{r\geq 0}\sum_{i=2}^n\lambda_i^r \notag\\
     &=(n-1)+  \tr(B) + \sum_{r\ge2}\tr(B^r) \notag\\  
    &=n-2+\sum_{r\ge2}\tr(B^r).          \label{eq:K-expansion}
\end{align}
The quadratic moment in \eqref{eq:K-expansion} (i.e., the summand of \eqref{eq:K-expansion} for $r=2$) is the first moment that changes under an edge perturbation and will provide the leading term {of the change}. The cubic moment in \eqref{eq:K-expansion} will also be treated explicitly to reveal a cancellation at order $(\avdeg)^{-2}$. As we shall see {in Section~\ref{sec_higher_order_tail}},  all moments of order at least four can then be controlled collectively by operator-norm estimates. 

When an edge is added, 
the changes in the summands of \eqref{eq:K-expansion}  
for $r=2$ and $r=3$ can be computed directly by considering weighted closed walks in $G$ of lengths two and three, in the following sense. 
A \emph{closed walk of length $r$} is a sequence
\[
v_0,v_1,\ldots,v_r=v_0
\]in which every two consecutive vertices are adjacent in $G$. The \emph{$S$-weight} 
of this walk is the product 
$S_{v_0v_1}S_{v_1v_2}\cdots S_{v_{r-1}v_0}$; note that $\operatorname{tr}(S^r)
=
\sum_{v_0,\ldots,v_{r-1}}
S_{v_0v_1}S_{v_1v_2}\cdots S_{v_{r-1}v_0}$,  where the sum is taken over all closed walks of length $r$.
{
Note that, for every $r\geq 1$, $\spec(S^r)=(1,\lambda_2^r,\dots,\lambda_n^r)$ and $\spec(B^r)=(0,\lambda_2^r,\dots,\lambda_n^r)$, so
\begin{equation}\label{eq:B-S-moments}
  \tr(B^r)=\tr(S^r)-1.
\end{equation}
}

\section{Uniform estimates in the dense regime}
\label{sec:dense-uniform-estimates}

Fix $p_0\in(0,1)$.  Throughout the dense-regime argument, we let
$p=p(n)\in[p_0,1]$ and sample $G\sim G(n,p)$.
The purpose of the rest of this section  is to prove the following result.

\begin{theorem}
\label{thm:dense-exclusion}
Whp, $G$ is connected and both
of the following hold:
\begin{enumerate}[label=\textnormal{(\roman*)},leftmargin=2.2em]
  \item $\K(G+uv)<\K(G)$ for every nonedge $uv$ of $G$;
  \item $G-e$ is connected and $\K(G-e)>\K(G)$ for every edge
  $e$ of $G$.
\end{enumerate}
\end{theorem}

We start with the following lemma, which gives uniform estimates on the degrees, number of common neighbours, number of triangles containing a fixed vertex, and operator norm of $B$.
Assertions~(i)--(iii)
are standard consequences of Chernoff bounds for
binomial random variables (see, for example,
\cite[Theorem~2.1]{JansonLuczakRucinski}), while (iv) uses the self-adjoint matrix Bernstein inequality
\cite[Theorem~1.4]{Tropp}. 
For completeness, we give a self-contained proof in Appendix~\ref{appendix_omitted_proofs}.

In the following, we let $\avdeg:=p(n-1)$
be the expected vertex degree, and set
$\varepsilon_n:=\sqrt{\frac{\log n}{n}}$.
 $\varepsilon_n$ should be thought of as the dense-regime concentration scale. 

\begin{lemma}
\label[lemma]{lem:event}
There exists a constant $C=C(p_0)>0$ such that the following holds. For $n\geq 3$ and $p=p(n)\in[p_0,1]$, let $\mathcal E_{n,p}$ be the event that $G\sim G(n,p)$ simultaneously satisfies the following conditions:
\begin{enumerate}[label=\textnormal{(\roman*)},leftmargin=2.2em]
  \item $\left|d_v/\avdeg-1\right|\leq C\varepsilon_n$ for every vertex $v$;%
  \item for every two distinct vertices $u,v$,
  \[
    \left|\frac{|N(u)\cap N(v)|}{(n-2)p^2}-1\right|\leq C\varepsilon_n;
  \]
  \item if $t_v$ is the number of triangles in $G$ containing $v$, then
  \[
    \left|\frac{t_v}{\binom{n-1}{2}p^3}-1\right|\leq C\varepsilon_n
    \quad\text{for every }v;
  \]
  \item $G$ is connected and $\opnorm{B}\leq C\varepsilon_n$.
\end{enumerate}
Then, 
\begin{align*}
    \lim_{n\to \infty}\inf_{p\in[p_0,1]}\Pp(\mathcal E_{n,p})=1.
\end{align*}
\end{lemma}

We henceforth work on the event $\mathcal E_{n,p}$ in Lemma~\ref{lem:event}.  Because the same
deterministic constant $C=C(p_0)$ works for \textit{every} vertex and pair, all assertions in the remaining part of this section
hold simultaneously for all nonedges; in particular, the nonedge may be
chosen after observing $G$.

\subsection{The two leading spectral moments}
\label{sec_leading_spectral_moments}
Fix $u,v \in V(G)$. For later use, we introduce the following quantities. 
\begin{align*}
  R_u&:=\sum_{w\sim u}\frac{1}{d_w},\\
  C_{uv}&:=\sum_{w\in N(u)\cap N(v)}\frac{1}{d_w},\\
  Q_u&:=\sum_{\substack{\{w,w'\}\subseteq N(u)\\w\sim w'}}
          \frac{1}{d_wd_{w'}}. 
\end{align*}
The next result follows from~\Cref{lem:event} by replacing every degree in the corresponding sum by
$\avdeg$, and then using the degree, codegree, and triangle estimates.

\begin{lemma}\label[lemma]{lem:local}
On the event $\mathcal E_{n,p}$ of Lemma~\ref{lem:event}, 
for every pair of distinct vertices $u,v$,
\begin{align}
  R_u&=1+O(\varepsilon_n),                  \label{eq:R}\\
  C_{uv}&=p\frac{n-2}{n-1}+O(\varepsilon_n),                   \label{eq:C}\\
  Q_u&=\frac{p}{2}\frac{n-2}{n-1}+O(\varepsilon_n).            \label{eq:Q}
\end{align}
\end{lemma}

\begin{remark}
In~\Cref{lem:local} and henceforth, all constants hidden in the asymptotic notation $O(\cdot)$ may depend on $p_0$, but not on the
particular choice of $p\in[p_0,1]$.
In addition, all constants are deterministic and independent of $u$ and $v$;
in particular, the three estimates in~\Cref{lem:local} hold simultaneously for all pairs.
\end{remark}

\begin{proof}[Proof of~\Cref{lem:local}] 
By Lemma~\ref{lem:event}(i), for every vertex $w\in V(G)$, its degree satisfies
$d_w=\avdeg(1+O(\varepsilon_n))$, and hence also
$d_w^{-1}=\avdeg^{-1}(1+O(\varepsilon_n))$, with one deterministic constant
serving for all vertices.  Next, we combine such factors using that, as $\varepsilon_n \in o(1)$, 
\begin{align}
(1+O(\varepsilon_n))(1+O(\varepsilon_n))=1+O(\varepsilon_n).
\label{eq:O}
\end{align}

Fix $u,v \in V(G)$.
The sum $R_u=\sum_{w\sim u}d_w^{-1}$ has $d_u$ many $d_w^{-1}$ terms, each equal to
$\avdeg^{-1}(1+O(\varepsilon_n))$, and therefore 
\[
  R_u=d_u\,\avdeg^{-1}\bigl(1+O(\varepsilon_n)\bigr)
     =1+O(\varepsilon_n).
\]

The sum $C_{uv}=\sum_{w\in N(u)\cap N(v)}d_w^{-1}$ has $|N(u)\cap N(v)|$ many $d_w^{-1}$ terms, and by
Lemma~\ref{lem:event}(ii) we obtain
\[
  C_{uv}
  =(n-2)p^2\,\avdeg^{-1}\bigl(1+O(\varepsilon_n)\bigr)
  =p\frac{n-2}{n-1}\bigl(1+O(\varepsilon_n)\bigr)
  =p\frac{n-2}{n-1}+O(\varepsilon_n),
\]
where the last step uses that the main term is bounded by $p\le1$, so the
multiplicative error may be absorbed into an additive one.

Finally, the summands of $Q_u=\sum_{\substack{\{w,w'\}\subseteq N(u)\\w\sim w'}}d_w^{-1}d_{w'}^{-1}$ are indexed by the edges inside $N(u)$,
that is, by the $t_u$ triangles containing $u$, and each summand equals
$\avdeg^{-2}(1+O(\varepsilon_n))$.  Lemma~\ref{lem:event}(iii) therefore gives
\[
  Q_u
  =\binom{n-1}{2}p^3\,\avdeg^{-2}\bigl(1+O(\varepsilon_n)\bigr)
  =\frac{p}{2}\frac{n-2}{n-1}\bigl(1+O(\varepsilon_n)\bigr)
  =\frac{p}{2}\frac{n-2}{n-1}+O(\varepsilon_n).
\]

Each step invoked only the estimates of Lemma~\ref{lem:event}, whose
constants are deterministic and uniform over vertices and pairs; hence a
single constant works for all $u,v$ simultaneously.
\end{proof}

Henceforth, for a parameter $F$ of the graph $G$, we let $F'$ be the corresponding parameter of the graph obtained from $G$ after adding a nonedge $uv$. Also, we write
$\Delta F=F'-F$. 
We find the following.

\begin{lemma}
\label[lemma]{lem_1523_1008}
On the event $\mathcal E_{n,p}$ of Lemma~\ref{lem:event}, 
uniformly over nonedges
 $uv$ it holds that 
\begin{align}
    \Delta\tr(S^2)
    &=-\frac{2}{p^2n^2}
      +O\!\left(\frac{\sqrt{\log n}}{n^{5/2}}\right)
      \label{eq:delta2}
      \intertext{and}
      \Delta\tr(S^3)
  &=O\!\left(\frac{\sqrt{\log n}}{n^{5/2}}\right).
  \label{eq:delta3}
\end{align}
\end{lemma}

\begin{proof}
Fix a nonedge $uv$.
The normalised adjacency matrix $S$ and the transition matrix $T$ of $G=(V,E)$ satisfy
\begin{align}
\label{eqn_2035_27_09}
 \tr(S^2)=\tr(T^2)
  =
  \sum_{w,w'\in V}T_{ww'}T_{w'w}
  =\sum_{\substack{w,w'\in V\\w\sim w'}}\frac{1}{d_wd_{w'}}
  =
  2\sum_{ww'\in E}\frac{1}{d_wd_{w'}}.
\end{align}
Hence, the quantity $\Delta \tr(S^2)$ only depends on the degrees of vertices $u$ and $v$ and those of their neighbours. In particular, we obtain the following identity.
\begin{align}\label{eq:delta2-exact}
  \Delta\tr(S^2)
  &=
\frac{2}{(d_u+1)(d_v+1)}+2\sum_{w\sim u}\frac{1}{d_w}\left(\frac{1}{d_u+1}-\frac{1}{d_u}\right)
+
2\sum_{w\sim v}\frac{1}{d_w}\left(\frac{1}{d_v+1}-\frac{1}{d_v}\right)\notag \\
&=\frac{2}{(d_u+1)(d_v+1)}
   -\frac{2R_u}{d_u(d_u+1)}
   -\frac{2R_v}{d_v(d_v+1)}.
\end{align}
Note that $d_u = \avdeg(1+O(\varepsilon_n))$ (\Cref{lem:event}) and $1/\avdeg=o(\varepsilon_n)$; thus, 
$$d_u+1 = \avdeg\left(1+O(\varepsilon_n)+\frac{1}{\avdeg}\right) = \avdeg(1 + O(\varepsilon_n)),$$ and
similarly $d_v+1 = \avdeg(1 + O(\varepsilon_n))$.
Hence,
\begin{align*}
    \frac{1}{(d_u+1)(d_v+1)}
    &=
    \frac{1}{\avdeg^2}+O\left(\frac{\varepsilon_n}{\avdeg^2}\right).
\end{align*}
By \eqref{eq:R}, $R_u =1+O(\varepsilon_n)$ and $R_v =1+O(\varepsilon_n)$, and thus 
\begin{align*}
    \frac{R_u}{d_u(d_u+1)}
    & = \frac{1}{\avdeg^2}+O\left(\frac{\varepsilon_n}{\avdeg^2}\right),\quad\mbox{and}\\
    \frac{R_v}{d_v(d_v+1)}
    &=
    \frac{1}{\avdeg^2}+O\left(\frac{\varepsilon_n}{\avdeg^2}\right).
\end{align*}
Combining this with~\eqref{eq:delta2-exact}, we deduce that 
\begin{align*}
\Delta\operatorname{tr}(S^2)
& {=
2\left(\frac{1}{\bar d^2}
+O\!\left(\frac{\varepsilon_n}{\bar d^2}\right)\right)
-2\left(\frac{1}{\bar d^2}
+O\!\left(\frac{\varepsilon_n}{\bar d^2}\right)\right)
-2\left(\frac{1}{\bar d^2}
+O\!\left(\frac{\varepsilon_n}{\bar d^2}\right)\right)}\\
&=
-\frac{2}{\bar d^2}
+O\!\left(\frac{\varepsilon_n}{\bar d^2}\right) \\
      &=
      -\frac{2}{p^2(n-1)^2}+O\!\left(\frac{\sqrt{\log n}}{n^{5/2}}\right)\\
    &=-\frac{2}{p^2n^2}
      +O\!\left(\frac{\sqrt{\log n}}{n^{5/2}}\right).         
\end{align*}
Similarly,
\[
  \tr(S^3)=\tr(T^3)=6\sum_{\{w,w', w''\}\text{ a triangle}}
               \frac{1}{d_wd_{w'}d_{w''}}.
\]
When comparing $\tr(S'^3)$ with $\tr(S^3)$, we note that the new triangles are precisely $uvw$ with
$w\in N(u)\cap N(v)$, while each old triangle through $u$ or $v$
has its weight reduced.
More in detail, we have
\begin{align*}
    \Delta\tr(S^3)
    &=
    \sum_{w\in N(u)\cap N(v)}\frac{6}{d_w(d_u+1)(d_v+1)}\\
    &\quad+\sum_{\substack{\{w,w'\}\subseteq N(u)\\w\sim w'}}\frac{6}{d_wd_{w'}}\left(\frac{1}{d_u+1}-\frac{1}{d_u}\right)\\
    &\quad+\sum_{\substack{\{w,w'\}\subseteq N(v)\\w\sim w'}}\frac{6}{d_wd_{w'}}\left(\frac{1}{d_v+1}-\frac{1}{d_v}\right).
\end{align*}
Recalling the quantities defined at the beginning of~\Cref{sec_leading_spectral_moments}, we deduce that
\begin{equation}\label{eq:delta3-exact}
  \Delta\tr(S^3)
  =6\left[
       \frac{C_{uv}}{(d_u+1)(d_v+1)}
       -\frac{Q_u}{d_u(d_u+1)}
       -\frac{Q_v}{d_v(d_v+1)}
     \right].
\end{equation}
Set $n'\coloneqq\frac{n-2}{n-1}$. 
{We will use that, since $\varepsilon_n=o(1)$ {and by~\eqref{eq:O}},
\[
\frac{1}{1+O(\varepsilon_n)}
=1+O(\varepsilon_n).
\]}
Substituting \eqref{eq:C}--\eqref{eq:Q} yields the following:
\begin{align*}
    \frac{C_{uv}}{(d_u+1)(d_v+1)}
    &=
    \frac{pn'+O(\varepsilon_n)}
     {\avdeg^2+\avdeg^2O(\varepsilon_n)}
=
\frac{pn'+O(\varepsilon_n)}
     {\avdeg^2(1+O(\varepsilon_n))}
     = \frac{pn'+O(\varepsilon_n)}
     {\avdeg^2}(1+O(\varepsilon_n))
=
\frac{pn'}{\avdeg^2}
+O\left(\frac{\varepsilon_n}{\avdeg^2}\right),
\\
    \frac{Q_u}{d_u(d_u+1)}
    &=\frac{\frac{p}{2}n'+O(\varepsilon_n)}{\avdeg^2+\avdeg^2 O(\varepsilon_n)}
    =
    \frac{pn'}{2\avdeg^2}+O\left(\frac{\varepsilon_n}{\avdeg^2}\right),
    \intertext{and, similarly,}
    \frac{Q_v}{d_v(d_v+1)}
    &=\frac{\frac{p}{2}n'+O(\varepsilon_n)}{\avdeg^2+\avdeg^2 O(\varepsilon_n)}
    =
    \frac{pn'}{2\avdeg^2}+O\left(\frac{\varepsilon_n}{\avdeg^2}\right).
\end{align*}
Therefore, the first-order terms in~\cref{eq:delta3-exact} cancel out, and we obtain
\begin{equation*}
  \Delta\tr(S^3)
  =O\!\left(\frac{\varepsilon_n}{\avdeg^2}\right)
  =O\!\left(\frac{\sqrt{\log n}}{n^{5/2}}\right).
\end{equation*}
This concludes the proof of the lemma.
\end{proof}

\subsection{The higher-order tail}
\label{sec_higher_order_tail}
At this point, it remains to show that the sum of the changes in all spectral moments of
order at least four  {in~\eqref{eq:K-expansion}} is smaller than the main term in \eqref{eq:delta2}. 

Fix a nonedge $uv$. Since $S = B+\Pi$, 
{and by~\eqref{eq:Pi}} we have $B \Pi = (S-\Pi) \Pi = \Pi - \Pi = 0$. Hence,
for every $r \geq 1$, 
$$S^r = (B+\Pi)^r = B^r + \Pi^r = B^r + \Pi.$$
Let
$D',S',\phi',\Pi'$, and $B'=S'-\Pi'$ be the matrices associated with
$G'=G+uv=(V,E')$. For example, we have $D'=\diag(d_i':i\in V)$ with $d'_i = d_i + \mathbf{1}_{\{i \in \{u,v\}\}}$ and  $\phi':=\frac{(D')^{1/2}\one}{\sqrt{2|E'|}}$.

As usual, we let $\Delta F=F'-F$ for any quantity $F$.
As $\tr(\Pi)=1$,
we have
$\Delta\tr(S^r)=\Delta\tr(B^r)$.
For $r=2,3$, we worked with $S$, since its traces admit convenient expressions in terms of edges and triangles of $G$. For the higher-order
tail, we return to $B$, whose powers decay because
$\opnorm{B}<1$.
In particular, our goal is now to prove the following. 

\begin{proposition}
\label[proposition]{lem:dense-higher-tail}
The following holds:
\begin{align}
  \sum_{r\ge4}
    \left|\Delta\tr(B^r)\right|
      =O\!\left(\frac{(\log n)^{3/2}}{n^{5/2}}\right).          \label{eq:tail}
\end{align}
\end{proposition}
In order to prove~\Cref{lem:dense-higher-tail}, it is crucial to upper bound the norm of the difference between $S$, $\phi$, and $\Pi$ and their counterparts $S'$, $\phi'$, and $\Pi'$ associated with the perturbed graph $G'$. We do so in~\Cref{lem_1523_1008_B},~\Cref{lem:1813_1607}, and~\Cref{lem_1524_1008_c}, respectively.
We start by defining a diagonal matrix $\Gamma$ by
\[
 \Gamma_{ii} :=
  \begin{cases}
    \sqrt{d_i/(d_i+1)} & \text{ if } i\in\{u,v\},\\
    1 & \text{ if } i\notin\{u,v\}.
  \end{cases}
\]

\begin{lemma}\label[lemma]{lem:S'}
The following equality holds:
\begin{equation}\label{eq:S-perturbation}
  S'=\Gamma S\Gamma
     +\frac{e_ue_v^{\top}+e_ve_u^{\top}}
            {\sqrt{(d_u+1)(d_v+1)}}.
\end{equation}
\end{lemma}
\begin{proof}
    For two vertices $i,j\in V(G)$, we find
    \begin{align*}
        S'_{i,j}
        &=
        \left\{
        \begin{array}{cll}
             S_{ij}&\mbox{ if }i,j\not\in\{u,v\}\mbox{ or }i=j,\\
             \frac{1}{\sqrt{(d_i+1)(d_j+1)}}&\mbox{ if }\{i,j\}=\{u,v\},\\ 
             \sqrt{\frac{d_i}{d_i+1}}\,S_{ij}&\mbox{ if }i\in\{u,v\}, j\not\in\{u,v\},\\ 
             \sqrt{\frac{d_j}{d_j+1}}\,S_{ij}&\mbox{ if }i\not\in\{u,v\}, j\in\{u,v\}.
        \end{array}
        \right.
    \end{align*}
    Then, a direct comparison shows that
$S'_{ij}=(\Gamma S\Gamma)_{ij}$ for every ordered pair
$(i,j)\notin\{(u,v),(v,u)\}$. The only missing entries are those
corresponding to the newly added edge $uv$.
    Adding the extra correction term $\frac{e_ue_v^{\top}+e_ve_u^{\top}}
            {\sqrt{(d_u+1)(d_v+1)}}$ gives the required identity~\eqref{eq:S-perturbation}.
\end{proof}

We first address the normalised adjacency matrix $S$.
If $A \in {\mathbb R}^{m \times n}$ is a matrix, then $\nucnorm{A}$ denotes the \emph{nuclear norm} of $A$ 
(see Glossary~\ref{sect:nuclear}).

\begin{lemma}
\label[lemma]{lem_1523_1008_B}
   $\nucnorm{S'-S}=O(n^{-1})$.
\end{lemma}

\begin{proof}
First, note that 
\begin{align*}
    \opnorm{\Gamma-I}=
    \max\left(1-\sqrt{\frac{d_u}{d_u+1}},1-\sqrt{\frac{d_v}{d_v+1}}\right).
\end{align*}
Recall that we are working in the event $\mathcal E_{n,p}$ of~\Cref{lem:event}. Applying part (i) of the lemma yields,
for $i\in\{u,v\}$,
\begin{align*}
    1-\sqrt{\frac{d_i}{d_i+1}}
 =
\frac{1-\frac{d_i}{d_i+1}}
     {1+\sqrt{\frac{d_i}{d_i+1}}}
= \frac{1}{(d_i+1)\left(1+\sqrt{d_i/(d_i+1)}\right)}
=
O(d_i^{-1})
=
O(\avdeg^{-1})
=
O(n^{-1}).
\end{align*}
It follows that $\opnorm{\Gamma-I}=O(n^{-1})$.
Observing that $\Gamma-I$ has rank 
two, it then follows (see Glossary~\ref{sect:nuclear}) that
\begin{align}
\label{eqn_1711_1607}
    \nucnorm{\Gamma-I}=O(n^{-1}).
\end{align}

Next, define $a := \frac{1}{\sqrt{(d_u+1)(d_v+1)}}$ 
and $Z\coloneqq a (e_ue_v^{\top}+e_ve_u^{\top})$, so that Lemma~\ref{lem:S'} gives  \begin{align}
S'=\Gamma S\Gamma+Z.
\label{eq:S'}
\end{align}
Note that $Z$ has eigenvalues $a$ and $-a$ with multiplicity 1, and that all other eigenvalues are $0$.
Using again part (i) of~\Cref{lem:event},
we compute
\begin{align}
\label{eqn_1712_1607}
    \nucnorm{Z}
    = {a + a}
    =
    \frac{2}{\sqrt{(d_u+1)(d_v+1)}}
    =O(\avdeg^{-1})
    =
    O(n^{-1}).
\end{align}
Observe that 
\begin{align}
  \Gamma S\Gamma-S=(\Gamma-I)S\Gamma+S(\Gamma-I). \label{eq:GammaSGamma}
\end{align}
Note also that the nuclear norm satisfies the mixed submultiplicative inequalities $\nucnorm{MN}\leq\opnorm{M}\nucnorm{N}$ and $\nucnorm{MN}\leq\nucnorm{M}\opnorm{N}$ (as special cases of~\cref{eqn_1608_2109}) as well as the triangle inequality $\nucnorm{M+N}\leq\nucnorm{M}+\nucnorm{N}$~\eqref{eq:nuclear-triangle}. 
Using~\eqref{eq:S'}  {and~\eqref{eq:GammaSGamma}}, 
we deduce that
\begin{align}
\label{eqn_27_09_1124}
\notag
    \nucnorm{S'-S}
    =
    \nucnorm{\Gamma S\Gamma-S+Z}
    & =
    \nucnorm{(\Gamma-I)S\Gamma+S(\Gamma-I)+Z} \\
    \notag
    &\leq
    \nucnorm{(\Gamma-I)S\Gamma}+\nucnorm{S(\Gamma-I)}+\nucnorm{Z}\\
    &\leq
    \nucnorm{\Gamma-I}\opnorm{S}\opnorm{\Gamma}+\opnorm{S}\nucnorm{\Gamma-I}+\nucnorm{Z}.
\end{align}
Note now that $\opnorm{S}=1$ since the spectrum of $S$ is the same as the spectrum of $T$ (recall the discussion in~\Cref{sec_spectral_expansion}). Also, clearly, $\opnorm{\Gamma}=1$.
Combining this with~\eqref{eqn_1711_1607} and~\eqref{eqn_1712_1607}, we deduce that
\begin{align*}
    \nucnorm{S'-S}
    &=O(n^{-1})+ O(n^{-1})+O(n^{-1})=O(n^{-1}),
\end{align*}
as required.
\end{proof}

Next, we turn to the vectors $\phi:=\frac{D^{1/2}\one}{\sqrt{2|E(G)|}}$ and $\phi':=\frac{(D')^{1/2}\one}{\sqrt{2|E(G')|}}$.

\begin{lemma}
\label[lemma]{lem:1813_1607}
    $\lVert\phi'-\phi\rVert_2=O(n^{-3/2})$.
\end{lemma}
\begin{proof}
Let $q := 2|E(G)| = \sum_{i \in V(G)} d_i$  
and $q':= 2|E(G')| = \sum_{i \in V(G')} d_i'$. Then we have $q'=q+2$, because $d'_i = d_i + \mathbf{1}_{\{i \in \{u,v\}\}}$.
Using~\Cref{lem:event}, we have $d_i=\Theta(n)$ uniformly in $i$, and consequently
$q=\Theta(n^2)$.
Set $\bold{w} := D^{1/2}\mathbf{1}$
    and
    $\bold{w}' := (D')^{1/2}\mathbf{1}$.
Since $\phi=\frac{1}{\sqrt{q}} \bold{w}$
and     $\phi'=\frac{1}{\sqrt{q+2}}\bold{w}'$,
we add and subtract
$\frac{1}{\sqrt{q+2}}\bold{w}$ to obtain
\begin{align}
    \phi'-\phi
    =
    \frac{1}{\sqrt{q+2}}\bold{w}'-\frac{1}{\sqrt{q}}\bold{w}\ & =
    \frac{1}{\sqrt{q+2}}\bold{w}'
    -\frac{1}{\sqrt{q+2}}\bold{w}
    +\frac{1}{\sqrt{q+2}}\bold{w}
    -\frac{1}{\sqrt{q}}\bold{w} \nonumber \\
    &=
    \frac{1}{\sqrt{q+2}}\left(\bold{w}'-\bold{w}\right)
    +
    \left(
        \frac{1}{\sqrt{q+2}}-\frac{1}{\sqrt{q}}
    \right)\bold{w}. \label{eq:delta-phi}
\end{align}
The vector $\bold{w}'-\bold{w}$ is supported only on $u$ and $v$, and
\begin{align*}
    \lVert \bold{w}'-\bold{w}\rVert_2^2
    &=
    \sum_{i\in\{u,v\}}
    \left(\sqrt{d_i+1}-\sqrt{d_i}\right)^2
    =
    \sum_{i\in\{u,v\}}
    \frac{1}{\left(\sqrt{d_i+1}+\sqrt{d_i}\right)^2}
    =
    O(n^{-1}),
\end{align*}
because $d_u,d_v=\Theta(n)$. Hence,
$\lVert \bold{w}'-\bold{w}\rVert_2=O(n^{-1/2})$. Since
$q=
\Theta(n^2)$, the {Euclidean} norm of the first term of~\eqref{eq:delta-phi}
is bounded as follows: 
\[
    \frac{\lVert \bold{w}'-\bold{w}\rVert_2}{\sqrt{q+2}}
    =
    O(n^{-3/2}).
\]

For the second term of~\eqref{eq:delta-phi}, using $\lVert \bold{w}\rVert_2=\sqrt{q}$, we obtain
\[
    \left|
        \frac{1}{\sqrt{q+2}}-\frac{1}{\sqrt{q}}
    \right|
    \lVert \bold{w}\rVert_2
    =
    1-\sqrt{\frac{q}{q+2}}
    =
    O(q^{-1})
    =
    O(n^{-2}).
\]
Therefore,
as required, 
    $\lVert\phi'-\phi\rVert_2
    \leq
    O(n^{-3/2})+O(n^{-2})
    =
    O(n^{-3/2})$.
\end{proof}

We finally turn to $\Pi:=\phi\phi^{\top}$ and $\Pi':=\phi'(\phi')^{\top}$.

\begin{lemma}
\label[lemma]{lem_1524_1008_c}
   $\nucnorm{\Pi'-\Pi}=O(n^{-3/2})$. 
\end{lemma}
\begin{proof}
    We find
    \begin{align*}
        \nucnorm{\Pi'-\Pi}
        &=
        \nucnorm{\phi'\phi'^\top-\phi\phi^\top}
        =
        \nucnorm{(\phi'-\phi)\phi'^\top+\phi(\phi'-\phi)^\top}
        \leq
        \nucnorm{(\phi'-\phi)\phi'^\top}+\nucnorm{\phi(\phi'-\phi)^\top}.
    \end{align*}
Since $\nucnorm{\bold{x}\bold{y}^\top}=\|\bold{x}\|_2\|\bold{y}\|_2$ for any two vectors $\bold{x},\bold{y}$  (see~\eqref{eq:nucl-eucl} in Glossary~\ref{glo:op}), we deduce that
    \begin{align*}
        \nucnorm{\Pi'-\Pi} 
        &\leq
        \|\phi'-\phi\|_2\left(\|\phi'\|_2+\|\phi\|_2\right).
        \end{align*}
        Using that $\|\phi\|_2=\|\phi'\|_2=1$, the result directly follows from~\Cref{lem:1813_1607}. 
\end{proof}

We now have all ingredients to prove the main result of this section.

\begin{proof}[Proof of~\Cref{lem:dense-higher-tail}]
    Let $W\coloneqq B'-B$. Then 
    \begin{align}
    \label{eq:E-nuclear}
        \nucnorm{W}
        &=
        \nucnorm{S'-\Pi'-S+\Pi}
        \leq\nucnorm{S'-S}+\nucnorm{\Pi'-\Pi}
        =
        O(n^{-1})+O(n^{-3/2})
        =
        O(n^{-1}).
    \end{align}

    Combining this with part (iv) of~\Cref{lem:event} and using that the operator norm is dominated by the nuclear norm (see Glossary~\ref{glo:op}), we deduce that
    \begin{align*}
        \opnorm{B'}
        =
        \opnorm{W+B}
        \leq\opnorm{W}+\opnorm{B}
        \leq \nucnorm{W}+\opnorm{B}
        =
        O(n^{-1})+O(\varepsilon_n)=O(\varepsilon_n),
    \end{align*}
    so
    \[
  \rho:=\max\{\opnorm{B},\opnorm{B'}\}=O(\varepsilon_n).
\]

For $r\geq 1$, the telescoping identity
\begin{align}
\label{eqn_1906_1607}
    (B')^r-B^r
    =
    \sum_{j=0}^{r-1}
        (B')^j W B^{r-1-j}
\end{align}
holds. Indeed, expanding the right-hand side gives
\[
\begin{aligned}
    \sum_{j=0}^{r-1}(B')^jWB^{r-1-j}
    &=
    \sum_{j=0}^{r-1}
    \left(
        (B')^{j+1}B^{r-1-j}
        -
        (B')^jB^{r-j}
    \right) \\
    &=
    (B')^r-B^r,
\end{aligned}
\]
since all the intermediate terms cancel.
We then obtain
\begin{align}
\label{eqn_2026_27_09}
\notag
    \left|
    \operatorname{tr}\bigl((B')^r-B^r\bigr)
    \right|
    &\leq
    \sum_{j=0}^{r-1}
    \left|
        \operatorname{tr}
        \bigl((B')^jWB^{r-1-j}\bigr)
    \right| && \text{\eqref{eqn_1906_1607}} \\\notag
    &=
    \sum_{j=0}^{r-1}
    \left|
        \operatorname{tr}
        \bigl(B^{r-1-j}(B')^jW\bigr)
    \right| && \text{(cyclicity of the trace)} \\\notag
    &\leq
    \lVert W\rVert_*
    \sum_{j=0}^{r-1}
    \left\|
        B^{r-1-j}(B')^j
    \right\|_{\mathrm{op}} && 
    \text{\eqref{eq:trace-duality}} \\\notag
    &\leq
    \lVert W\rVert_*
    \sum_{j=0}^{r-1}
        \lVert B\rVert_{\mathrm{op}}^{r-1-j}
        \lVert B'\rVert_{\mathrm{op}}^j && \text{(submultiplicativity of the operator norm)} \\
    &\leq
    r\lVert W\rVert_*\rho^{r-1} && \text{(definition of $\rho$).}
\end{align}
Therefore,
\begin{align*}
    \sum_{r\geq 4}
    \left|
        \Delta\operatorname{tr}(B^r)
    \right|
&=
\sum_{r\geq 4}
    \left|\tr((B')^r-B^r)\right|
    \leq
    \lVert W\rVert_*
    \sum_{r\geq 4}r\rho^{r-1}. 
    \end{align*}
    Observe now that $\sum_{r\geq 4}r\rho^{r-1}$ is the derivative of a tail of the geometric series.
    Since $\rho=O(\varepsilon_n)=o(1)$, we have $|\rho|<1$ for all
sufficiently large $n$, and thus
\begin{align}
\label{eqn_2033_27_09}
    \sum_{r \geq 4} r \rho^{r-1} = \frac{d}{d\rho} \sum_{r \geq 4} \rho^{r} = \frac{d}{d\rho} \left( \frac{\rho^4}{1-\rho} \right) = \frac{4\rho^3(1-\rho) + \rho^4}{(1-\rho)^2} = \frac{\rho^3 (4 - 3\rho)}{(1-\rho)^2}.
\end{align}    
    Hence,
\begin{align*} 
\sum_{r\geq 4}
    \left|
        \Delta\operatorname{tr}(B^r)
    \right|
&\leq
    \lVert W\rVert_*
    \frac{\rho^3(4-3\rho)}{(1-\rho)^2}
=
    O\bigl(\lVert W\rVert_*\rho^3\bigr) 
    =
    O\left(\frac{\varepsilon_n^3}{n}\right) 
    =
    O\left(\frac{(\log n)^{3/2}}{n^{5/2}}\right),
\end{align*}
and the proof is concluded.
\end{proof}

\subsection{Proving~\Cref{thm:dense-exclusion}}

We can now prove that whp the Kemeny--Braess paradox does not occur in the dense Erd\H{o}s--R\'enyi regime.

\begin{proof}[Proof of~\Cref{thm:dense-exclusion}]
We shall assume that the event $\mathcal E_{n,p}$ of~\Cref{lem:event} holds. By part (iv) of the lemma,
we have that $G$ is connected and
$\opnorm{B}=O(\varepsilon_n)<1$ for all sufficiently large $n$.
In particular, $G$ is nonbipartite  {(see~\eqref{eq:non-bip} in Glossary~\ref{glo:op})}; clearly, the same holds for the graph $G+uv$ obtained from $G$ by adding a nonedge $uv$. 
Consequently, the nonprincipal
normalised-adjacency eigenvalues of both graphs have modulus less than one.
Hence, the expansion
\eqref{eq:K-expansion} is valid for both graphs.
Applying \eqref{eq:K-expansion} to $G+uv$ and to $G$, and subtracting
the resulting 
expansions, we obtain
\[
    \K(G+uv)-\K(G)
    =
    \sum_{r\geq2}\Delta\tr(B^r).
\]
By~\eqref{eq:B-S-moments}, we have 
$\Delta\tr(B^r)=\Delta\tr(S^r)$ for every $r\geq1$.

All estimates used below hold on the event $\mathcal E_{n,p}$
with deterministic constants independent of the vertices involved.
Since the subsequent calculations are deterministic, their implied
constants are likewise independent of the chosen nonedge $uv$.
Equations \eqref{eq:delta2},
\eqref{eq:delta3}, and \eqref{eq:tail} give, uniformly over all nonedges,
\[
  \K(G+uv)-\K(G)
  =-\frac{2}{p^2n^2}
   +O\!\left(\frac{(\log n)^{3/2}}{n^{5/2}}\right).
\]
Hence, the change in Kemeny's constant
is negative for every nonedge when $n$ is large enough.  This proves part~(i) of~\Cref{thm:dense-exclusion}.

We now prove part~(ii), continuing to work on the event $\mathcal E_{n,p}$.  Fix an edge $e=uv$ of $G$ and put
$H=G-e$. We first verify that $H$ satisfies conditions (i)–(iv) of~\Cref{lem:event} with $C$ replaced by a possibly larger constant $C'$, depending only on $p_0$ and independent of the deleted edge $e$. 

Deleting $e$ from $G$ changes each degree by at most one. Since
$\avdeg=p(n-1)=\Theta(n)$ and $\avdeg\varepsilon_n=\Theta(\sqrt{n\log n})$,
the degree (in $H$) of each vertex $v$ satisfies
\[
  d_H(v)=\avdeg\bigl(1+O(\varepsilon_n)\bigr)
\]
with an implied constant independent of the edge $e$. 
Similarly, the common neighbourhood (in $H$) of every pair of two distinct vertices $u,v$ satisfies
\[
  |N_H(u)\cap N_H(v)|
  =(n-2)p^2\bigl(1+O(\varepsilon_n)\bigr).
\]

Let $t_G(v)$ and $t_H(v)$ denote the numbers of triangles containing
$v$ in the two graphs $G$ and $H$, respectively. If $i\notin\{u,v\}$, deleting $uv$ destroys at
most one triangle containing $i$. If $i\in\{u,v\}$, it destroys exactly
$|N_G(u)\cap N_G(v)|=O(n)$ such triangles. Since
$\binom{n-1}{2}p^3=\Theta(n^2)$, it follows that, uniformly in $v$ and
$e$,
\[
  t_H(v)=\binom{n-1}{2}p^3\bigl(1+O(\varepsilon_n)\bigr).
\]
Thus, the first three facts of Lemma~\ref{lem:event} are stable under the
deletion of an arbitrary edge.

It remains to check (iv). Associate the matrices 
$S_X$,
$\Pi_X$, and $B_X$ from Section~\ref{sec_spectral_expansion} with a
graph $X$.
Note that the proofs of~\Cref{lem_1523_1008_B},~\Cref{lem:1813_1607}, and~\Cref{lem_1524_1008_c}
only use that all degrees are $\Theta(n)$, 
a property that, as we have just established, also holds for $H$.
Hence, we can apply exactly the same arguments to the extension
$H+e=G$ 
to find that
$\nucnorm{S_G-S_H}=O(n^{-1})$ and 
$\nucnorm{\Pi_G-\Pi_H}=O(n^{-3/2})$
uniformly over $e$.
Consequently,
\begin{align*}
  \opnorm{B_H}
  &\le \opnorm{B_G}+\opnorm{B_H-B_G}\\
  &\le \opnorm{B_G}
      +\opnorm{S_H-S_G}
      +\opnorm{\Pi_H-\Pi_G} \\
  &\le \opnorm{B_G}
      +\nucnorm{S_H-S_G}
      +\nucnorm{\Pi_H-\Pi_G}\\
  &=O(\varepsilon_n).
\end{align*}
In particular, for all sufficiently large $n$ we have $\opnorm{B_H}<1$, which means that $H$ is connected and non-bipartite. 

The proofs of \eqref{eq:delta2}, \eqref{eq:delta3}, and
\eqref{eq:tail} use only that $G$ satisfies the four conditions of~\Cref{lem:event}. We may therefore repeat them with $H$ as the base graph and
with its missing edge $e$ added back. Since $H+e=G$, we deduce that
\[
  \K(G)-\K(G-e)
  =-\frac{2}{p^2n^2}
   +O\!\left(\frac{(\log n)^{3/2}}{n^{5/2}}\right),
\]
where, once again, the constant in the asymptotic notation is independent of $e$. 
We deduce that
$\K(G-e)>\K(G)$ for every edge $e$ when $n$ is sufficiently large on
the event of \Cref{lem:event}.  All constants and all probability estimates
used above depend only on $p_0$, so the conclusion is uniform over
$p\in[p_0,1]$.  Since the event has probability tending to one, the proof
is complete.
\end{proof}

\section{The two threshold theorems}
\label{sec:logarithmic-sparse-regime}
The uniform estimates of~\Cref{sec:dense-uniform-estimates} are enough to prevent the occurrence of Kemeny--Braess paradox in the dense regime, thus showing that the Kemeny analogues of {Conjecture}~\ref{conj_chung} and \Cref{thm_eldan} fail. However, in order to obtain our main results~\Cref{thm:edge-threshold} and~\Cref{thm:nonedge-threshold} covering the whole phase diagram (see~\Cref{fig:braess-phases})
we need to go beyond the dense regime.
In particular, as a first step, we need a slight modification of the results in~\Cref{sec_higher_order_tail} to capture the case when the parameter $p=p(n)$ may tend to zero as $n$ grows to infinity. This is done in the next subsection.

\subsection{A sparse edge-perturbation formula}
As before, for a graph $X$ and a vertex $v\in V(X)$, let $t_X(v)$ denote the number of triangles containing
$v$ in $X$. For $d>0$ and a vertex $v$ of positive degree, we define the
\emph{harmonic neighbour statistic} by
\begin{equation}\label{eq:sparse-harmonic-statistic}
  h_{X,d}(v)
  :=
  \frac{d}{d_X(v)}
  \sum_{w\sim_X v}\frac1{d_X(w)}.
\end{equation}
Here $d$ is a deterministic scale parameter; in the random-graph
applications below, we take $d=\avdeg$, the expected degree of the ambient
random graph.\footnote{We denote the graph by $X$ instead of $G$ here to stress that it is now viewed as a deterministic graph.} We write $h_X(v)$ when $d$ is clear from the context.

\begin{definition}
\label[definition]{def:kemeny-tame}
For each $n\in \mathbb{N}$, let $\mathcal F_n$ be a family of triples $(X,u,v)$,
where $X$ is a connected graph on $n$ vertices and $uv$ is a nonedge
of $X$. Assume $d=d(n)$ tends to $\infty$ as $n\to \infty$.
We say that $\mathcal F=(\mathcal F_n)_n$ is
\emph{Kemeny-tame at scale $d$} if there exists a constant $c>0$
such that the following conditions hold uniformly over
$(X,u,v)\in\mathcal F_n$:
\begin{enumerate}
  \item[$(i)$]
  $d_X(u)=\Theta(d)$ and $d_X(v)=\Theta(d)$.

  \item[$(ii)$]
  $h_{X,d}(u)=1+o(1)$ and $h_{X,d}(v)=1+o(1)$.

  \item[$(iii)$]
  $\opnorm{B_X}=O(d^{-1/2})$.

  \item[$(iv)$]
  $|N_X(u)\cap N_X(v)|=o(d)$.

  \item[$(v)$]
  $t_X(u)+t_X(v)=o(d^2)$.

  \item[$(vi)$]
  For all sufficiently large $n$, every common neighbour of $u$ and $v$,
  and every vertex belonging to a triangle containing $u$ or $v$,
  has degree at least $cd$ in $X$.
\end{enumerate}
All asymptotic statements refer to $n\to\infty$, with $d=d(n)$.
\end{definition}

The proof of the following proposition closely follows the argument used in the proof of~\Cref{lem_1523_1008} and~\Cref{lem:dense-higher-tail}, though with different asymptotic conditions given by Kemeny-tameness. For completeness, we give a full proof below.

\begin{proposition}
\label[proposition]{prop:sparse-addition}
Let $\mathcal F=(\mathcal F_n)_n$ be Kemeny-tame at scale $d$.
Then, uniformly over $(X,u,v)\in\mathcal F_n$,
\begin{align}
  \K(X+uv)-\K(X)
  =
  2\left[
    \frac{1}{(d_X(u)+1)(d_X(v)+1)}
    -\frac{1}{d\cdot(d_X(u)+1)}
    -\frac{1}{d\cdot(d_X(v)+1)}
  \right]
  +o(d^{-2}).
  \label{eq:sparse-addition}
\end{align}
\end{proposition}

\begin{proof}
Fix $(X,u,v)\in\mathcal F_n$, and write $X^+=X+uv$.
We now proceed as in~\Cref{sec_higher_order_tail}, by considering the diagonal matrix $\Gamma$ given by 
 {$\Gamma_{ii}:=\sqrt{\frac{d_X(i)}{d_X(i)+1}}$ for $i \in \{u,v\}$, and 1 if $i \notin \{u,v\}$.}
%,
%
Setting
\[
  Z:=
  \frac{e_ue_v^\top+e_ve_u^\top}
       {\sqrt{(d_X(u)+1)(d_X(v)+1)}},
\]
we have 
$S_{X^+}=\Gamma S_X\Gamma+Z$.
The matrix $\Gamma-I$ has rank at most two
and, by condition~(i), operator norm $O(d^{-1})$.
Since $\opnorm{S_X}=1$, $\opnorm{\Gamma}\le1$, and
$\nucnorm{Z}=O(d^{-1})$ (cf.~\cref{eqn_1712_1607}), the same calculation as in~\cref{eqn_27_09_1124} gives
\begin{equation}\label{eq:edge-deletion-S-bound}
  \nucnorm{S_{X^+}-S_X}=O(d^{-1}).
\end{equation}
Let $q=2|E(X)|$ and write $\bold{w} := D_X^{1/2}\mathbf{1}$
    and
    $\bold{w}' := (D_{X^+})^{1/2}\mathbf{1}$, as in the proof of~\Cref{lem:1813_1607}. 
We have from~\cref{eq:delta-phi} that 
\begin{align*}
\phi_{X^+}-\phi_X=\frac{1}{\sqrt{q+2}}\left(\bold{w}'-\bold{w}\right)
    +
    \left(
        \frac{1}{\sqrt{q+2}}-\frac{1}{\sqrt{q}}
    \right)\bold{w}
    \end{align*}
and thus,
since $q\ge(d_X(u)+d_X(v))=\Omega(d)$,
\[
  \lVert\phi_{X^+}-\phi_X\rVert_2
  =
  O\left(\frac1{\sqrt{qd(n)}}+\frac1q\right)
  =O(d^{-1}).
\]
Consequently, reasoning as in~\Cref{lem_1524_1008_c},
\begin{equation}\label{eq:edge-deletion-Pi-bound}
  \nucnorm{\Pi_{X^+}-\Pi_X}
  \le 2\lVert\phi_{X^+}-\phi_X\rVert_2
  =O(d^{-1}),
\end{equation}
and hence, via~\cref{eq:E-nuclear},
\[
  \nucnorm{B_{X^+}-B_X}=O(d^{-1}).
\]
Condition~(iii) therefore gives
\begin{equation}\label{eq:edge-deletion-B_X+-bound}
  \opnorm{B_{X^+}}
  \le
  \opnorm{B_X}+\nucnorm{B_{X^+}-B_X}
  =O(d^{-1/2}).
\end{equation}

In particular, since $d(n)\to\infty$, condition~(iii) and
\eqref{eq:edge-deletion-B_X+-bound} imply that
$\opnorm{B_X},\opnorm{B_{X^+}}<1$ for all sufficiently large $n$.
Thus~\eqref{eq:K-expansion} applies to both graphs.
Reasoning as in~\cref{eqn_2026_27_09}, we obtain
\[
  \left|\tr(B_{X^+}^r)-\tr(B_X^r)\right|
  \le
  r\nucnorm{B_{X^+}-B_X}
  \max\{\opnorm{B_X},\opnorm{B_{X^+}}\}^{r-1}.
\]
Summing over $r\ge4$ and using~\cref{eqn_2033_27_09}, we find
\begin{equation}\label{eq:sparse-higher-tail}
  \sum_{r\ge4}
  \left|\tr(B_{X^+}^r)-\tr(B_X^r)\right|
  =O(d^{-5/2})
  =o(d^{-2}).
\end{equation}

For the quadratic moment, we know from~\cref{eqn_2035_27_09} that
\[
  \tr(S_X^2)
  =
  2\sum_{ww'\in E(X)}\frac1{d_X(w)d_X(w')}
\]
and, thus, by~\cref{eq:delta2-exact},
\begin{align}
  \tr(S_{X^+}^2)-\tr(S_X^2)
  ={}&
  \frac{2}{(d_X(u)+1)(d_X(v)+1)}
  \notag\\
  &-\frac{2}{d_X(u)(d_X(u)+1)}
  \sum_{w\sim_X u}\frac1{d_X(w)}
  \notag\\
  &-\frac{2}{d_X(v)(d_X(v)+1)}
  \sum_{w\sim_X v}\frac1{d_X(w)}.
  \label{eq:sparse-quadratic-exact}
\end{align}
By condition~(ii) and~\eqref{eq:sparse-harmonic-statistic},
the two sums are
\[
  (1+o(1))\frac{d_X(u)}{d(n)}
  \qquad\text{and}\qquad
  (1+o(1))\frac{d_X(v)}{d(n)},
\]
respectively.
Substituting in \eqref{eq:sparse-quadratic-exact}, together with
condition~(i), gives the main term in \eqref{eq:sparse-addition}
with an error $o(d^{-2})$.
For the cubic moment,~\cref{eq:delta3-exact} gives
\begin{align}
  \tr(S_{X^+}^3)-\tr(S_X^3)
  ={}&
  \frac{6}{(d_X(u)+1)(d_X(v)+1)}
  \sum_{w\in N_X(u)\cap N_X(v)}\frac1{d_X(w)}
  \notag\\
  &-\frac{6}{d_X(u)(d_X(u)+1)}
  \sum_{\substack{ww'\in E(X)\\w,w'\in N_X(u)}}
       \frac1{d_X(w)d_X(w')}
  \notag\\
  &-\frac{6}{d_X(v)(d_X(v)+1)}
  \sum_{\substack{ww'\in E(X)\\w,w'\in N_X(v)}}
       \frac1{d_X(w)d_X(w')}.
  \label{eq:sparse-cubic-exact}
\end{align}
By conditions~(i) and~(vi), there is a constant $C>0$,
independent of $(X,u,v)$, such that
\[
  \left|\tr(S_{X^+}^3)-\tr(S_X^3)\right|
  \le
  C\left[
    \frac{|N_X(u)\cap N_X(v)|}{d(n)^3}
    +
    \frac{t_X(u)+t_X(v)}{d(n)^4}
  \right]
  =o(d^{-2}),
\]
where the last equality uses conditions~(iv) and~(v).
Combining the quadratic and cubic estimates with
\eqref{eq:B-S-moments}, \eqref{eq:K-expansion}, and
\eqref{eq:sparse-higher-tail} proves \eqref{eq:sparse-addition}.
\end{proof}

We immediately have the following corollary.

\begin{corollary}
\label[corollary]{cor:sparse-deletion}
Let $\mathcal F=(\mathcal F_n)_n$ be Kemeny-tame at scale $d$.
For each $(X,u,v)\in\mathcal F_n$,
set $G=X+uv$, and write
$a=d_G(u)$, $b=d_G(v)$.
Then, uniformly over $(X,u,v)\in\mathcal F_n$,
\begin{align}
  \K(G-uv)-\K(G)
  &=
  \frac{2(a+b-d)}{abd}+o(d^{-2}).
  \label{eq:sparse-deletion}
\end{align}
\end{corollary}

\begin{proof}
{Applying \Cref{prop:sparse-addition} with $X=G-uv$ and simplifying} terms gives
\begin{align*} 
 \K(G-uv)-\K(G) 
  &=
  -2\left[
    \frac{1}{{ab}}
    -\frac{1}{{a}d}
    -\frac{1}{{b}d}
  \right]
  +o(d^{-2}) \\
  & =
2\left[\frac{a+b-d}{abd} \right]+o(d^{-2}). \qedhere 
\end{align*}

\end{proof}

\subsection{Random-graph estimates}
\label{sec:sparse-random-graph-estimates}
Throughout this subsection, $G\sim G(n,p)$ with $p=p(n)=o(1)$, and
$\avdeg=(n-1)p$ satisfies $\avdeg-\log n\to\infty$; in particular,
$\avdeg\ge\log n$ for all sufficiently large $n$.
(\Cref{lem:sparse-barriers} and its proof hold for arbitrary $p=p(n)\in(0,1)$.)
We write $h_G(v)$ for the harmonic neighbour statistic
$h_{G,\avdeg}(v)$ at scale $\avdeg$; see~\eqref{eq:sparse-harmonic-statistic}.
 
Every bound below that holds whp comes with deterministic constants and
deterministic rates (such as $\xi_n$ and $\zeta_n$), which do not depend on
the vertices involved. This is what makes \Cref{prop:sparse-addition} and
\Cref{cor:sparse-deletion} applicable. For instance, in the proof of
\Cref{thm:nonedge-threshold}(ii), let $\mathcal F_n$ consist of all triples
$(X,u,v)$ such that $uv\notin E(X)$ and $X$ is a graph on $[n]$ satisfying
the bounds in \Cref{lem:sparse-uniform-estimates}(i)--(ii),
$\min_wd_X(w)>(\frac12+\eta)\avdeg$ and $\max_w|h_X(w)-1|\le\zeta_n$, with
$\eta$ from \Cref{lem:sparse-barriers}(i) and $\zeta_n=\zeta_n(\frac12)$ from
\Cref{lem:sparse-uniform-estimates}(iii). Then $\mathcal F=(\mathcal F_n)_n$
is Kemeny-tame at scale $\avdeg$, and whp $(G,u,v)\in\mathcal F_n$ for every
nonedge $uv$ of $G$. 
The application to the Kemeny--Braess edge result (\Cref{thm:edge-threshold}) is similar, with 
the Kemeny-tame family consisting of triples $(G-uv,u,v)$.

Consider the function $I$ given by
\begin{equation}\label{eq:sparse-rate-function}
  I(x)=1-x+x\log x,\qquad x>0.
\end{equation}
Since $I'(x)=\log x$, the function $I$ is
strictly decreasing on $(0,1]$ and strictly increasing on $[1,\infty)$.
Moreover, $2I(1/2)=1-\log2$; hence, the parameter $C^*$ in~\Cref{thm:edge-threshold} and~\Cref{thm:nonedge-threshold} satisfies $C^*=1/(2I(1/2))$.
Let $Y$ be a binomial random variable with mean $\mu$. Chernoff's inequality
\cite[Theorem~2.1]{JansonLuczakRucinski} gives
\begin{equation}\label{eq:sparse-chernoff}
  \Pp(Y\le x\mu)\le e^{-\mu I(x)}\quad(0<x\le1),
  \qquad
  \Pp(Y\ge x\mu)\le e^{-\mu I(x)}\le\Bigl(\frac ex\Bigr)^{x\mu}\quad(x\ge1),
\end{equation}
and
\begin{equation}\label{eq:sparse-bernstein}
  \Pp\bigl(|Y-\mu|\ge t\bigr)
  \le2\exp\Bigl(-\frac{t^2}{2(\mu+t/3)}\Bigr),
  \qquad t\ge0.
\end{equation}

We shall use the following result, due to Hoffman, Kahle, and Paquette.

\begin{theorem}
[\protect{\cite[Theorem~1.1]{HoffmanKahlePaquette}}]
\label{theorem_HKP}
Fix $\delta>0$ and suppose that
\[
p\geq \left(\frac12+\delta\right)\frac{\log n}{n}.
\]
Take $G\sim G(n,p)$ and let $\tilde G$ be a largest connected component of $G$. For every fixed
$\varepsilon>0$, there is a constant $C_0=C_0(\delta,\varepsilon)$ such that
every nonzero
eigenvalue $\nu$ of the normalised Laplacian of $\tilde G$ satisfies
\[
|1-\nu|<C_0\avdeg^{-1/2}
\]
with probability at least
\[
1-C_0n\exp\bigl(-(2-\varepsilon)\avdeg\bigr)
  -C_0\exp\bigl(-\avdeg^{1/4}\log n\bigr).
\]
\end{theorem}

The following lemma collects the global estimates needed to apply
\Cref{prop:sparse-addition} and \Cref{cor:sparse-deletion}.
 
\begin{lemma}\label[lemma]{lem:sparse-uniform-estimates}
There are an absolute constant $C\ge1$ and a deterministic sequence
$\xi_n\to0$ such that the following hold.
\begin{enumerate}[label=\textnormal{(\roman*)},leftmargin=2.2em]
  \item Whp, 
   $G\sim G(n,p)$ is connected and $\opnorm{B_G}< C\avdeg^{-1/2}$.
  \item Whp,
  \[
    \max_vd_G(v)\le C\avdeg,\qquad
    \max_{u\ne v}|N_G(u)\cap N_G(v)|\le\xi_n\avdeg,\qquad
    \max_vt_G(v)\le\xi_n\avdeg^2.
  \]
  \item For every fixed $c>0$, there is a deterministic sequence
  $\zeta_n=\zeta_n(c)\to0$ such that, whp, the following holds:
  if $\min_vd_G(v)\ge c\avdeg$, then $|h_G(v)-1|\le\zeta_n$ for every
  vertex $v$.
\end{enumerate}
\end{lemma}

\begin{proof}
(i) Connectivity is the classical theorem of Erd\H{o}s and
R\'enyi~\cite{ErdosRenyi}. For the spectral bound, we apply
\Cref{theorem_HKP} with its parameters set to
$\delta=\frac14$ and $\varepsilon=\frac12$. This is allowed because
$p\ge\frac34\cdot\frac{\log n}{n}$ for all large $n$. The failure probability
given there is then at most $C_0ne^{-3\avdeg/2}+C_0e^{-\avdeg^{1/4}\log n}$,
which tends to $0$ because $\avdeg\ge\log n$. Thus, whp, every nonzero
eigenvalue $\nu$ of the normalised Laplacian of the giant component satisfies
$|1-\nu|< C_0\avdeg^{-1/2}$. On the connectivity event the giant component
is $G$ itself, and the eigenvalues of $B_G$ are $0$ and
$1-\nu_j(\mathcal L_G)$ for $2\le j\le n$ (see
\Cref{sec_spectral_expansion}), where $\mathcal L_G$ is the normalised Laplacian of $G$. Hence, $\opnorm{B_G}< C_0\avdeg^{-1/2}$.
 
(ii) Let $C:=\max\{e^2,C_0\}$. Each degree $d_G(v)$ is
$\operatorname{Bin}(n-1,p)$ with mean $\avdeg\ge\log n$. So, for each vertex $v$ of $G$,
\eqref{eq:sparse-chernoff} with $x=e^2$ gives
$\Pp(d_G(v)\ge e^2\avdeg)\le e^{-I(e^2)\avdeg}=e^{-(e^2+1)\avdeg}\le n^{-3}$,
and a union bound over $v$ gives the degree bound.
 
Next, put $L:=\log(1/p)$, so that $L\to\infty$, and $\xi_n:=L^{-1/2}$. For
$u\ne v$, $|N_G(u)\cap N_G(v)|$ is
$\operatorname{Bin}(n-2,p^2)$, whose mean is upper bounded by $\avdeg p\le\xi_n\avdeg$.
By \eqref{eq:sparse-chernoff},
\[
  \Pp\bigl(|N_G(u)\cap N_G(v)|\ge\xi_n\avdeg\bigr)
  \le\Bigl(\frac{e\avdeg p}{\xi_n\avdeg}\Bigr)^{\xi_n\avdeg}
  =\exp\bigl(-\avdeg\,\xi_n(L+\log\xi_n-1)\bigr)
  \le n^{-3}
\]
for all large $n$. Here we used $\xi_n(L+\log\xi_n-1)=(1-o(1))L^{1/2}\to\infty$
and $\avdeg\ge\log n$. A union bound over the fewer than $n^2$ pairs gives
the second bound of (ii).
 
Finally, conditionally on $N_G(v)$, the number $t_G(v)$ of edges inside
$N_G(v)$ is $\operatorname{Bin}(\binom{d_G(v)}2,p)$. On the event
$\{d_G(v)\le C\avdeg\}$, it is therefore stochastically dominated by 
$\operatorname{Bin}(\lfloor C^2\avdeg^2/2\rfloor,p)$, whose mean is at most
$C^2\avdeg^2p/2$. Using~\cref{eq:sparse-chernoff} again, we get
\[
  \Pp\bigl(d_G(v)\le C\avdeg,\ t_G(v)\ge\xi_n\avdeg^2\bigr)
  \le\Bigl(\frac{eC^2p}{2\xi_n}\Bigr)^{\xi_n\avdeg^2}
  \le n^{-3}
\]
for all large $n$. A union bound over $v$ completes the proof of (ii).
 
(iii) Fix $c>0$. Replacing $c$ by $\min\{c,1\}$, we may assume that
$c\le1$. Define $\eta_n:=\avdeg^{-1/4}$ and $M_n:=\lceil9\avdeg^{1/2}\rceil$.
 
\emph{Step 1: few atypical neighbours.} Fix a vertex $v$, and condition on
$N_G(v)=\mathcal N$ with $|\mathcal N|\le C\avdeg$. For $w\in\mathcal N$, let
$q_w:=1+\bigl|N_G(w)\setminus(\mathcal N\cup\{v\})\bigr|$,
so that
$d_G(w)=q_w+|N_G(w)\cap N_G(v)|$.
Conditionally on $\mathcal N$, the variables $q_w-1$, $w\in\mathcal N$, are
independent $\operatorname{Bin}(n-|\mathcal N|-1,p)$ variables: they are
determined by disjoint sets of pairs, none of which contains $v$. Put
$\ell:=1+(n-|\mathcal N|-1)p=\avdeg+1-|\mathcal N|p$, so that
$\E q_w=\ell$. Since $|\mathcal N|\le C\avdeg$,
\begin{equation}\label{eq:sparse-ell-bounds}
  \avdeg(1-Cp)\le\ell\le\avdeg+1.
\end{equation}
Call $w\in\mathcal N$ \emph{atypical} (for $v$) if
$|q_w-\ell|>\eta_n\ell$. By~\eqref{eq:sparse-bernstein} and
\eqref{eq:sparse-ell-bounds}, for all large $n$,
\[
  \Pp(w\text{ is atypical}\mid\mathcal N)
  \le2\exp\Bigl(-\frac{\eta_n^2\ell}{3}\Bigr)
  \le2e^{-\avdeg^{1/2}/4}=:\pi_n.
\]
Once we fix $N_G(v)=\mathcal N$, the events “$w$ is atypical” are independent for different $w\in\mathcal N$. Hence, the probability that at least $M_n$ vertices
of $\mathcal N$ are atypical is at most
\[
  \binom{|\mathcal N|}{M_n}\pi_n^{M_n}
  \le(C\avdeg\,\pi_n)^{M_n}
  =\exp\Bigl(-M_n\Bigl(\frac{\avdeg^{1/2}}4-\log(2C\avdeg)\Bigr)\Bigr)
  \le e^{-2\avdeg}
  \le n^{-2}
\]
for all large $n$. (For the second-to-last inequality, note that
$M_n\avdeg^{1/2}/4\ge\frac94\avdeg$ and $M_n\log(2C\avdeg)=o(\avdeg)$.)
Averaging over $\mathcal N$ and taking a union bound over $v$ shows the
following: whp, every vertex of degree at most $C\avdeg$ has fewer than
$M_n$ atypical neighbours. Let us call this event $\mathcal E$.
 
\emph{Step 2: the harmonic statistic.} We now work on the intersection of the
event $\mathcal E$ with the events in~(ii), and assume that
$\min_wd_G(w)\ge c\avdeg$. Fix $v$ and write $\mathcal N=N_G(v)$, so
that $c\avdeg\le|\mathcal N|\le C\avdeg$. If $w\in\mathcal N$ is typical,
then
\begin{align}
\label{eqn_2342_2809}
  (1-\eta_n)\ell\le q_w\le d_G(w)\le q_w+\xi_n\avdeg
  \le(1+\eta_n)\ell+\xi_n\avdeg,
\end{align}
where the third inequality uses that $d_G(w)=q_w+|N_G(v)\cap N_G(w)|$ and $|N_G(v)\cap N_G(w)|\leq\xi_n\avdeg$ by (ii).
Every $w\in\mathcal N$ satisfies $d_G(w)\ge c\avdeg$. Moreover,
$|\mathcal N|=d_G(v)\ge c\avdeg$, and on $\mathcal E$ fewer than $M_n$
vertices of $\mathcal N$ are atypical. For each typical neighbour $w$,~\cref{eqn_2342_2809} and~\eqref{eq:sparse-ell-bounds} give
\[
  (1-\eta_n)\avdeg(1-Cp)
  \le d_G(w)
  \le (1+\eta_n)(\avdeg+1)+\xi_n\avdeg.
\]
Consequently,
\[
  \frac{1}{(1+\eta_n)(1+\avdeg^{-1})+\xi_n}
  \le \frac{\avdeg}{d_G(w)}
  \le \frac{1}{(1-\eta_n)(1-Cp)}.
\]

We now lower bound $h_G(v)$ by retaining only the contributions
of its typical neighbours.
There are at least $|\mathcal N|-M_n$ such neighbours, so
\begin{align*}
  h_G(v)
  &=\frac{1}{|\mathcal N|}
    \sum_{w\in\mathcal N}\frac{\avdeg}{d_G(w)}\\
  &\ge
    \frac{1-M_n/|\mathcal N|}
         {(1+\eta_n)(1+\avdeg^{-1})+\xi_n}\\
  &\ge
    \frac{1-M_n/(c\avdeg)}
         {(1+\eta_n)(1+\avdeg^{-1})+\xi_n},
\end{align*}
where the last inequality uses $|\mathcal N|\ge c\avdeg$.

Next, we upper bound $h_G(v)$ as follows. Note that $v$ has at most $|\mathcal N|$ typical
neighbours, each contributing at most
$1/((1-\eta_n)(1-Cp))$ to the sum of the scaled reciprocals in $h_G(v)$.
Each atypical neighbour contributes at most $1/c$, because
$d_G(w)\ge c\avdeg$, and there are fewer than $M_n$ of them.
Therefore,
\begin{align*}
  h_G(v)
  &\le \frac{1}{|\mathcal N|}
    \left(
      \frac{|\mathcal N|}{(1-\eta_n)(1-Cp)}
      +\frac{M_n}{c}
    \right)\\
  &=\frac{1}{(1-\eta_n)(1-Cp)}
    +\frac{M_n}{c|\mathcal N|}\\
  &\le\frac{1}{(1-\eta_n)(1-Cp)}
    +\frac{M_n}{c^2\avdeg},
\end{align*}
again using $|\mathcal N|\ge c\avdeg$.
Both bounds on $h_G(v)$ are deterministic, do not depend on $v$, and tend to $1$.
Taking for $\zeta_n$ the larger of their distances from $1$ proves~(iii).
\end{proof}

The next lemma provides vertices and edges with the degree and
neighbourhood properties needed for the lower-threshold arguments.
Whp there are many vertices, respectively edges, whose degrees
equal a prescribed value $k=\lfloor\alpha\avdeg\rfloor$, and whose
neighbourhoods are as regular as \Cref{def:kemeny-tame} requires. 
In the following, 
we write
$b_N(j):=\Pp(\operatorname{Bin}(N,p)=j)$.
 
\begin{lemma}\label[lemma]{lem:sparse-low-degree-reservoirs}
Fix $\alpha\in(0,1)$ and $\delta>0$, and set $k:=\lfloor\alpha\avdeg\rfloor$.
\begin{enumerate}[label=\textnormal{(\roman*)},leftmargin=2.2em]
  \item Suppose that $n\exp(-\avdeg I(\alpha))\ge n^\delta$ for all large
  $n$. Then, whp, there are two nonadjacent vertices $u,v$ without common
  neighbours such that, for each $w\in\{u,v\}$,
  \[
    d_G(w)=k,\qquad t_G(w)=0,\qquad|h_G(w)-1|\le2\avdeg^{-1/3}.
  \]
  \item Suppose that $n\exp(-2\avdeg I(\alpha))\ge n^\delta$ for all large
  $n$. Then, whp, there is an edge $uv$ of $G$ with the following
  properties: $d_G(u)=d_G(v)=k$; neither $u$ nor $v$ lies in a triangle
  (in particular, $u$ and $v$ have no common neighbour); and, for $H=G-uv$,
  \[
    |h_H(u)-1|\le2\avdeg^{-1/3}\qquad\text{and}\qquad
    |h_H(v)-1|\le2\avdeg^{-1/3}.
  \]
\end{enumerate}
\end{lemma}

\begin{proof}
Since $I(\alpha)>0$, either hypothesis gives $\avdeg\le\log n/I(\alpha)$.
Thus $\avdeg=O(\log n)$,
$p=O(\log n/n)$, and $k\to\infty$. 
(Recall that we are assuming $\avdeg-\log n\to\infty$; hence, we in fact have $\avdeg=\Theta(\log n)$.) 

For integers $N\ge j\ge1$ we have the
identities
\begin{equation}\label{eq:sparse-binomial-ratios}
  \frac{b_N(j-1)}{b_N(j)}=\frac{j(1-p)}{(N-j+1)p},
  \qquad
  \frac{b_{N-1}(j)}{b_N(j)}=\frac{N-j}{N(1-p)}.
\end{equation}
For $N=n-O(1)$ and $j=k-O(1)$, the first ratio tends to $\alpha$ and the
second tends to $1$. Hence, $b_{n-r}(k-s)/b_{n-1}(k)\to\alpha^s$ for all fixed
integers $r\ge1$ and $s$. 

We now estimate $b_{n-1}(k)$ in terms of $I(\alpha)$ and the expected degree $\avdeg$. 
We first approximate it using the
Poisson distribution $\operatorname{Po}(Np)$, cf.~\cite[Chapter~VI]{Feller1968}. 
Write $\operatorname{po}_{Np}(j):=\Pp(\operatorname{Po}(Np)=j)
=e^{-Np}(Np)^j/j!$, and note that
\[
  b_N(j)
  =\frac{N(N-1)\cdots(N-j+1)}{j!}\,p^j(1-p)^{N-j}
  =\operatorname{po}_{Np}(j)\cdot
  e^{Np}(1-p)^{N-j}\cdot
   \prod_{i=0}^{j-1}\Bigl(1-\frac iN\Bigr).
\]
Since $\prod_i(1-x_i)\ge1-\sum_ix_i$ whenever all $x_i\in[0,1]$, the
product over $i$ lies between $1-j(j-1)/(2N)\ge1-j^2/(2N)$ and $1$. 
As $p=o(1)$, for $n$ large enough we have that $0\leq p\leq\frac{1}{2}$.
Since, in this case, $-p-p^2\le\log(1-p)\le-p$, 
it holds that
$e^{-Np^2}\le e^{Np}(1-p)^N\le e^{Np}(1-p)^{N-j}\le e^{jp}$. Consequently,
\begin{equation}\label{eq:sparse-poisson-comparison}
  \Bigl(1-\frac{j^2}{2N}\Bigr)e^{-Np^2}
  \le\frac{b_N(j)}{\operatorname{po}_{Np}(j)}
  \le e^{jp}.
\end{equation}
We now apply \eqref{eq:sparse-poisson-comparison} with
$N=n-1$ and $j=k$, so that $Np=\avdeg$. Since $k\le\avdeg$, each of
$k^2/(2N)$, $Np^2$, and $kp$ is at most $\avdeg^2/(n-1)=O((\log n)^2/n)$,
which tends to $0$. Therefore,
\begin{align}
\label{eqn_2202_2909}
  b_{n-1}(k)=\bigl(1+o(1)\bigr)\operatorname{po}_{\avdeg}(k).
\end{align}
 
The problem is now to estimate $\operatorname{po}_{\avdeg}(k)$. To that end, we use Stirling's formula with
the explicit error bounds of Robbins~\cite{robbins1955remark} (see
also~\cite[Chapter~II]{Feller1968}): for every integer $m\ge1$,
\begin{equation}\label{eq:stirling-robbins}
  \sqrt{2\pi m}\,\Bigl(\frac me\Bigr)^m e^{1/(12m+1)}
  <m!<
  \sqrt{2\pi m}\,\Bigl(\frac me\Bigr)^m e^{1/(12m)}.
\end{equation}
By~\eqref{eq:sparse-rate-function},
$\avdeg-k-k\log(\avdeg/k)=\avdeg I(k/\avdeg)$. Applying
\eqref{eq:stirling-robbins} with $m=k$, and using $e^{1/(12k+1)}>1$, 
we obtain
\begin{align}
\label{eqn_2203_b_2909}
  \frac{e^{-1/(12k)}}{\sqrt{2\pi k}}\,e^{-\avdeg I(k/\avdeg)}
  <\operatorname{po}_{\avdeg}(k)=\frac{e^{-\avdeg}\avdeg^k}{k!}
  <\frac{1}{\sqrt{2\pi k}}\,e^{-\avdeg I(k/\avdeg)}.
\end{align}
Finally, we replace $k/\avdeg$ by $\alpha$. We have
$\alpha\avdeg-1<k\le\alpha\avdeg$ and, once $\avdeg\ge2/\alpha$, also
$k\ge\alpha\avdeg/2$; thus $\pi\alpha\avdeg\le2\pi k\le2\pi\alpha\avdeg$.
Moreover, $I$ is decreasing on $(0,1]$, and $|I'(x)|=|\log x|\le\log(2/\alpha)$
for $x\in[\alpha/2,\alpha]$. The mean value theorem therefore gives
\begin{align}
\label{eqn_eqn_2203_c_2909}
  0\le\avdeg\bigl(I(k/\avdeg)-I(\alpha)\bigr)
  \le(\alpha\avdeg-k)\log\frac2\alpha
  \le\log\frac2\alpha.
\end{align}
Combining~\cref{eqn_2202_2909},~\cref{eqn_2203_b_2909}, and~\cref{eqn_eqn_2203_c_2909}, we conclude that, for all sufficiently large $n$,
\begin{equation}\label{eq:sparse-local-degree}
  \frac{\sqrt\alpha}{8}\,\avdeg^{-1/2}e^{-\avdeg I(\alpha)}
  \le b_{n-1}(k)\le
  \frac{1}{\sqrt\alpha}\,\avdeg^{-1/2}e^{-\avdeg I(\alpha)};
\end{equation}
in particular, 
\begin{align}
\label{eqn_1928_3009}
b_{n-1}(k)=\Theta\bigl(\avdeg^{-1/2}e^{-\avdeg I(\alpha)}\bigr),
\end{align}
with implicit constants depending only on $\alpha$.
 
\emph{Part~(i).} Let $V_k$ be the set of vertices of degree $k$. By
\eqref{eq:sparse-local-degree}, the hypothesis that $n\exp(-\avdeg I(\alpha))\ge n^\delta$ for all large $n$, and $\avdeg=O(\log n)$, we have that
\begin{align}
\label{eqn_1722_3009}
\mathbb E|V_k|
=n b_{n-1}(k)
\ge \frac{\sqrt\alpha}{8}\bar d^{-1/2}
n e^{-\bar d I(\alpha)}
\ge \frac{\sqrt\alpha}{8}\bar d^{-1/2}n^\delta
\geq n^{\delta/2}
\end{align}
for all large $n$. Now pick two vertices $u\ne v$.
Conditioning on whether $uv\in E(G)$ makes the remaining degrees of $u$
and $v$ independent $\operatorname{Bin}(n-2,p)$ variables. So, with
$A:=b_{n-2}(k-1)$ and $B:=b_{n-2}(k)$,
\[
  \Pp(d_G(u)=k)=pA+(1-p)B
  \qquad\text{and}\qquad
  \operatorname{Cov}\bigl(\one_{\{d_G(u)=k\}},\one_{\{d_G(v)=k\}}\bigr)
  =p(1-p)(A-B)^2.
\]
Since $A,B=\Theta(b_{n-1}(k))$, the covariance is $O(p\,b_{n-1}(k)^2)$.
To bound the variance of $|V_k|$, write
$X_v:=\one_{\{d_G(v)=k\}}$, so that $|V_k|=\sum_vX_v$.
Since each $X_v$ is an indicator,
$\operatorname{Var}(X_v)=\E X_v(1-\E X_v)\le\E X_v$.
Therefore,
\begin{align*}
  \operatorname{Var}|V_k|
  &=
  \sum_v\operatorname{Var}(X_v)
  +\sum_{u\ne v}\operatorname{Cov}(X_u,X_v)
  \le
  \sum_v\E X_v+O\bigl(n^2p\,b_{n-1}(k)^2\bigr)
  =
  \E|V_k|+O(p)\bigl(\E|V_k|\bigr)^2,
\end{align*}
where we used that $\E|V_k|=n b_{n-1}(k)$.
Since $\E|V_k|\ge n^{\delta/2}\to\infty$ by~\cref{eqn_1722_3009} and
$p=O(\log n/n)\to0$, we obtain
\[
  \frac{\operatorname{Var}|V_k|}{(\E|V_k|)^2}
  \le \frac1{\E|V_k|}+O(p)=o(1).
\]
Thus $\operatorname{Var}|V_k|=o\bigl((\E|V_k|)^2\bigr)$.
By Chebyshev's inequality, whp
\begin{align}
\label{eqn_1849_30_09}
|V_k|\ge\frac12\E|V_k|.
\end{align}

Set
$\vartheta:=\avdeg^{-1/3}$, and
call a vertex $w$ \emph{good} if $d_G(w)=k$, $t_G(w)=0$, and
$|d_G(w')-\avdeg|\le\vartheta\avdeg$ for every $w'\sim w$. Conditionally on
$d_G(w)=k$, the neighbourhood $N_G(w)$ is a uniformly random $k$-subset
$\mathcal N$ of $V\setminus\{w\}$, independent of the edges not containing
$w$. Hence, by the union bound,
$\Pp(t_G(w)>0\mid d_G(w)=k)\le\binom k2p=O(\avdeg^2p)=o(1)$.

Condition, in addition, on $N_G(w)=\mathcal N$ and $t_G(w)=0$.
Then each $w'\in\mathcal N$ has degree $d_G(w')=1+Y_{w'}$ with
$Y_{w'}\sim\operatorname{Bin}(n-k-1,p)$, and $|\E d_G(w')-\avdeg|\le1$ for large enough $n$.
For each $w'\in\mathcal N$, put
$\mu:=\E Y_{w'}=(n-k-1)p\le\avdeg$.
The triangle inequality gives
\[
  |d_G(w')-\avdeg|
  \le |d_G(w')-\E d_G(w')|
      +|\E d_G(w')-\avdeg|
  \le |Y_{w'}-\mu|+1.
\]
Consequently,
applying \eqref{eq:sparse-bernstein} to $Y_{w'}$ with $t=\vartheta\avdeg-1$,

\begin{align*}
  \Pp\bigl(|d_G(w')-\avdeg|>\vartheta\avdeg\bigr)
  &\le
  \Pp\bigl(|Y_{w'}-\mu|>\vartheta\avdeg-1\bigr)\\
  &\le
  2\exp\left(
    -\frac{(\vartheta\avdeg-1)^2}
          {2(\mu+(\vartheta\avdeg-1)/3)}
  \right)\\
  &\le 2\exp(-\vartheta^2\avdeg/3)
\end{align*}
for all sufficiently large $n$.
The last inequality follows from $\mu\le\avdeg$ and
\begin{align}
\label{eqn_1303_0210}
  \frac{(\vartheta\avdeg-1)^2}
       {2(\mu+(\vartheta\avdeg-1)/3)}
  \ge
  \vartheta^2\avdeg\,
  \frac{(1-1/(\vartheta\avdeg))^2}
       {2(1+\vartheta/3)}
  \ge \frac{\vartheta^2\avdeg}{3},
\end{align}
since the fraction in the middle expression tends to $1/2$ as
$\vartheta\to0$ and $\vartheta\avdeg=\avdeg^{2/3}\to\infty$.
A union bound over the $k$ vertices of $\mathcal N$ now gives
\begin{align*}
  &\Pp\bigl(\exists w'\in\mathcal N:
       |d_G(w')-\avdeg|>\vartheta\avdeg\bigr)
  \le
  2k\exp(-\vartheta^2\avdeg/3)
  =2k\exp(-\avdeg^{1/3}/3)
  =o(1).
\end{align*}

Let $Z$ be the number of vertices in $V_k$ that are not good.
The above argument shows that
$\Pp(w\text{ is good}\mid w\in V_k)=1-o(1)$, so $\E Z=o(\E|V_k|)$. 
Markov's inequality then gives
\[
  \Pp\left(Z>\frac14\E|V_k|\right)
  \le \frac{4\E Z}{\E|V_k|}
  =o(1).
\]
Thus, whp, at most $\frac14\E|V_k|$ vertices in $V_k$ are not good.
We also know from~\cref{eqn_1849_30_09} that
$|V_k|\ge\frac12\E|V_k|$ whp.
On the intersection of these events (which happens whp), the number of good vertices is
\[
  |V_k|-Z
  \ge \frac12\E|V_k|-\frac14\E|V_k|
  =\frac14\E|V_k|
  \ge\frac14n^{\delta/2},
\]
where we have used~\cref{eqn_1722_3009}.

Now, a good vertex $w$ satisfies
\[
  \frac1{1+\vartheta}\le h_G(w)=\frac\avdeg k\sum_{w'\sim w}\frac1{d_G(w')}
  \le\frac1{1-\vartheta},
\]
so $|h_G(w)-1|\le2\vartheta=2\avdeg^{-1/3}$. Moreover, at most
$\sum_{w'\sim w}d_G(w')\le k(1+\vartheta)\avdeg\le\avdeg^2$ vertices are at
distance $1$ or $2$ from a good vertex $w$. Since $\frac14n^{\delta/2}>\avdeg^2+1$ for
all large $n$, there must exist two good vertices at distance at least $3$.
These two vertices witness that condition (i) is satisfied.
 
\emph{Part~(ii).} Let $E_k$ be the set of edges $uv$ of $G$ with
$d_G(u)=d_G(v)=k$. Conditionally on $uv\in E(G)$, the degrees $d_G(u)-1$
and $d_G(v)-1$ are independent $\operatorname{Bin}(n-2,p)$ variables. Put
$\beta:=b_{n-2}(k-1)=\Theta(b_{n-1}(k))$ (where we have used the argument at the beginning of the proof). Then, by~\cref{eqn_1928_3009},
\[
  \E|E_k|=\binom n2p\beta^2=\Theta\bigl(ne^{-2\avdeg I(\alpha)}\bigr)
  \ge n^{\delta/2}
\]
for all large $n$, by the hypothesis of (ii).
The hypothesis also gives
$e^{-\avdeg I(\alpha)}\ge n^{-(1-\delta)/2}$, hence
$n\beta=\Theta(n\avdeg^{-1/2}e^{-\avdeg I(\alpha)})\ge n^{1/2}$ for all
large $n$.
 
We now bound $\operatorname{Var}|E_k|$.
For every $e\in\binom{[n]}2$, let
$X_e:=\mathbf1_{\{e\in E_k\}}$. Then $|E_k|=\sum_eX_e$, and 
\begin{align}
\label{eqn_1055_0210}
\notag
  \operatorname{Var}|E_k|
  &=\sum_e\operatorname{Var}(X_e)
    +\sum_{e\ne f}\operatorname{Cov}(X_e,X_f)\\
  &=\sum_e\operatorname{Var}(X_e)
    +\sum_{\substack{e\ne f\\|e\cap f|=1}}
      \operatorname{Cov}(X_e,X_f)
    +\sum_{\substack{e,f\\e\cap f=\varnothing}}
      \operatorname{Cov}(X_e,X_f).
\end{align}
Since the $X_e$ are indicator variables, we have
$\operatorname{Cov}(X_e,X_f)
  =\Pp(e\in E_k,\ f\in E_k)
   -\Pp(e\in E_k)\Pp(f\in E_k)$.
The diagonal terms satisfy
\[
  \sum_e\operatorname{Var}(X_e)
  =\sum_e\Pp(e\in E_k)\bigl(1-\Pp(e\in E_k)\bigr)
  \le\sum_e\Pp(e\in E_k)
  =
  \sum_e\E X_e
  =\E|E_k|.
\]
We now estimate the remaining terms in~\cref{eqn_1055_0210} separately, according to whether
the two candidate edges $e,f$ share one vertex or are disjoint.

First, let $u,v,w$ be distinct vertices.
Conditioning on whether $vw\in E(G)$ gives
\begin{align*}
\Cov(X_{uv},X_{uw})&\leq
  \Pp(uv\in E_k,\ uw\in E_k)\\
  &=p^2\,b_{n-3}(k-2)\bigl[(1-p)\,b_{n-3}(k-1)^2+p\,b_{n-3}(k-2)^2\bigr]\\
  &=O(p^2\beta^3).
\end{align*}
There are at most $n^3$ such ordered pairs, contributing
$O(n^3p^2\beta^3)=O((\E|E_k|)^2/(n\beta))$ to $\operatorname{Var}|E_k|$. 

Second, let $\{u,v\}$ and
$\{w,w'\}$ be disjoint, and condition on the four pairs between them. The
remaining degrees of $u,v,w,w'$ into the other $n-4$ vertices are then
independent $\operatorname{Bin}(n-4,p)$ variables. Write
$R_j:=b_{n-4}(k-1-j)$. The configuration without cross edges contributes
$p^2(1-p)^4R_0^4$ to $\Pp(uv\in E_k,\ ww'\in E_k)$. Each of the other
fifteen configurations contributes $O(p^3R_0^4)$, because $R_1,R_2=O(R_0)$ by the argument at the beginning of this proof.
Also, using~\cref{eq:sparse-binomial-ratios}, we get
\[
\frac{\beta}{R_0}
=
\frac{b_{n-2}(k-1)}{b_{n-3}(k-1)}
\frac{b_{n-3}(k-1)}{b_{n-4}(k-1)}
=(1+O(p))^2
=1+O(p)
\]
and, thus, 
$\beta=
R_0(1+O(p))$. Hence,
\begin{align*}
\Cov(X_{uv},X_{ww'})&=
  \Pp(uv\in E_k,\ ww'\in E_k)
   -\Pp(uv\in E_k)\,\Pp(ww'\in E_k)\\
  &
  =p^2(1-p)^4R_0^4+O(p^3R_0^4)-p^2\beta^4\\
  &
  =p^2R_0^4(1+O(p))-p^2\beta^4\\
  &
  =p^2\beta^4(1+O(p))-p^2\beta^4\\
  &
  =O(p^3\beta^4).
\end{align*}
Therefore, the at most $n^4$ such ordered pairs contribute $O(p)(\E|E_k|)^2$ to $\operatorname{Var}|E_k|$.

Altogether,~\cref{eqn_1055_0210} gives
\[
  \frac{\operatorname{Var}|E_k|}{(\E|E_k|)^2}
  =O\Bigl(\frac1{\E|E_k|}+\frac1{n\beta}+p\Bigr)=o(1).
\]
We now apply Chebyshev’s inequality to find that
\begin{align*}
    \Pp\left(|E_k|<\frac12\E|E_k|\right)
    \leq 
    \Pp\left(\big||E_k|-\E|E_k|\big|\geq \frac12\E|E_k|\right)
    \leq
    \frac{4\Var|E_k|}{(\E|E_k|)^2}=o(1),
\end{align*}
so $|E_k|\ge\frac12\E|E_k|$ whp.
 
It remains to show that a typical edge of $E_k$ has the
properties required by (ii). Conditional on $uv\in E_k$, the sets
$\mathcal N_u:=N_G(u)\setminus\{v\}$ and
$\mathcal N_v:=N_G(v)\setminus\{u\}$ are independent uniformly random
$(k-1)$-subsets of $V\setminus\{u,v\}$. The pairs inside
$V\setminus\{u,v\}$ are still present independently with probability $p$.
A triangle through $u$ or $v$ has one of three forms: $uvw$ with
$w\in\mathcal N_u\cap\mathcal N_v$; $uww'$ with $ww'\in E(G[\mathcal N_u])$;
or $vww'$ with $ww'\in E(G[\mathcal N_v])$. Hence the conditional
probability that $u$ or $v$ lies in a triangle is at most
\[
  \E|\mathcal N_u\cap\mathcal N_v|
  +\E e(G[\mathcal N_u])+\E e(G[\mathcal N_v])
  \le\frac{k^2}{n-2}+k^2p=o(1),
\]
where $e(-)$ denotes the number of edges.
Let $H=G-uv$. For $w\in\mathcal N_u\cup\mathcal N_v$ we have
\[
  d_H(w)=d_G(w)
  =\one_{\{w\in\mathcal N_u\}}+\one_{\{w\in\mathcal N_v\}}+Y_w,
\]
where $Y_w\sim\operatorname{Bin}(n-3,p)$ counts the edges from $w$ to
$V\setminus\{u,v,w\}$. Thus $|\E d_H(w)-\avdeg|\le2$.

For fixed $\mathcal N_u$ and $\mathcal N_v$, put
$\mu:=\E Y_w=(n-3)p$. Since $d_H(w)-Y_w$ is then deterministic,
we have $d_H(w)-\E d_H(w)=Y_w-\mu$. The triangle inequality therefore gives
\begin{align*}
  |d_H(w)-\avdeg|
  &\le |d_H(w)-\E d_H(w)|+|\E d_H(w)-\avdeg|
  \le |Y_w-\mu|+2.
\end{align*}
Consequently, proceeding as in the proof of part~(i), we apply
\eqref{eq:sparse-bernstein} with $t=\vartheta\avdeg-2$ to find
\begin{align*}
  \Pp\bigl(|d_H(w)-\avdeg|>\vartheta\avdeg\bigr)
  &\le \Pp\bigl(|Y_w-\mu|>\vartheta\avdeg-2\bigr)\\
  &\le
  2\exp\left(
    -\frac{(\vartheta\avdeg-2)^2}
          {2(\mu+(\vartheta\avdeg-2)/3)}
  \right)\\
  &\le 2\exp(-\avdeg^{1/3}/3)
\end{align*}
for all sufficiently large $n$ (where the last inequality is proved using the same argument as in~\cref{eqn_1303_0210}).
A union bound over the fewer than $2k$ vertices of
$\mathcal N_u\cup\mathcal N_v$ shows that the probability that some
such vertex $w$ satisfies $|d_H(w)-\avdeg|>\vartheta\avdeg$ is at most
$4k\exp(-\avdeg^{1/3}/3)=o(1)$.
This bound is uniform in the fixed neighbourhoods, so averaging over
them gives the same bound conditional only on $uv\in E_k$.

Hence, whp every
$w\in\mathcal N_u\cup\mathcal N_v$ satisfies
$(1-\vartheta)\avdeg\le d_H(w)\le(1+\vartheta)\avdeg$ and, thus,
\begin{align*}
  \frac1{1+\vartheta}
  \le\frac{\avdeg}{d_H(w)}
  \le\frac1{1-\vartheta}.
\end{align*}
Since $d_H(u)=|\mathcal N_u|=k-1$,
averaging these bounds over $w\in\mathcal N_u$ gives
\[
  \frac1{1+\vartheta}
  \le h_H(u)
  =\frac1{k-1}\sum_{w\in\mathcal N_u}\frac{\avdeg}{d_H(w)}
  \le\frac1{1-\vartheta}.
\]
Thus, for all sufficiently large $n$,
\[
  |h_H(u)-1|
  \le\frac{\vartheta}{1-\vartheta}
  \le2\vartheta,
\]
where the last inequality uses $\vartheta\le1/2$.
The same argument applies to $v$.

Consequently, the expected
number of edges of $E_k$ that fail one of the required properties of (ii) is
$o(\E|E_k|)$. Markov's inequality and $|E_k|\ge\frac12\E|E_k|$ then show
that, whp, some edge of $G$ satisfies all of them.
\end{proof}

We shall also need the following simple observation.
 
\begin{lemma}\label[lemma]{lem:sparse-barriers}
Fix $\varepsilon>0$.
\begin{enumerate}[label=\textnormal{(\roman*)},leftmargin=2.2em]
  \item If $\avdeg\ge(2C^*+\varepsilon)\log n$ for all large $n$, then there
  is $\eta=\eta(\varepsilon)>0$ such that, whp,
  $\min_vd_G(v)>(\frac12+\eta)\avdeg$.
  \item If $\avdeg\ge(C^*+\varepsilon)\log n$ for all large $n$, then there
  are $c=c(\varepsilon)>0$ and $\eta=\eta(\varepsilon)>0$ such that, whp,
  $\min_vd_G(v)\ge c\avdeg$ and $d_G(u)+d_G(v)>(1+\eta)\avdeg$ for every
  edge $uv$ of $G$.
\end{enumerate}
\end{lemma}

\begin{proof}
(i) Since $2C^*I(\frac12)=1$ and $I$ is continuous, we can fix
$\eta\in(0,\frac12)$ such that $\gamma:=(2C^*+\varepsilon)I(\frac12+\eta)>1$.
By \eqref{eq:sparse-chernoff} and a union bound,
\[
  \Pp\Bigl(\min_vd_G(v)\le\bigl(\tfrac12+\eta\bigr)\avdeg\Bigr)
  \le n\,e^{-\avdeg I(\frac12+\eta)}
  \leq n\,e^{-\gamma\log n}
  = n^{1-\gamma}\to0,
\]
where the second inequality uses the assumption on $\avdeg$.
 
(ii) Since $I(0^+)=1<C^*+\varepsilon$, we can fix $c\in(0,1)$ with
$(C^*+\varepsilon)I(c)>1$. The argument of~(i) then gives
$\min_vd_G(v)> c\avdeg$ whp. Next, fix $0<\eta<\eta'<\frac12$ such that
$\gamma:=2(C^*+\varepsilon)I(\frac{1+\eta'}2)>1$. Conditionally on
$uv\in E(G)$, the random variable $d_G(u)+d_G(v)-2$ is
$\operatorname{Bin}(2(n-2),p)$. 
For all large $n$,
its mean $\mu=2\avdeg\frac{n-2}{n-1}$
satisfies $\mu\ge2\avdeg-2$ and $(1+\eta)\avdeg\le\frac{1+\eta'}2\mu$. By using \eqref{eq:sparse-chernoff} and the fact that $I$ has values $\leq 1$ on $(0,1]$, we find
\begin{align*}
  &\Pp\bigl(d_G(u)+d_G(v)\le(1+\eta)\avdeg\bigm|uv\in E(G)\bigr)
  \le e^{-\mu I(\frac{1+\eta'}2)}\\
  &
\le e^{-(2\bar d-2)I(\frac{1+\eta'}2)}=e^{2I(\frac{1+\eta'}2)}
  e^{-2\bar d I(\frac{1+\eta'}2)}
\le e^2\,e^{-2\bar d I(\frac{1+\eta'}2)}=e^2\,e^{-\gamma\bar d/(C^*+\varepsilon)}.
\end{align*}

Since $\avdeg\ge(C^*+\varepsilon)\log n$ by assumption for large $n$, we have
$n\le e^{\avdeg/(C^*+\varepsilon)}$.
Take now a union bound over the $\binom n2$ pairs. The probability that some
edge violates the degree-sum bound is at most
\begin{align*}
e^2\binom n2p\,e^{-\gamma\avdeg/(C^*+\varepsilon)}
\le e^2\,n\avdeg\,e^{-\gamma\avdeg/(C^*+\varepsilon)}
\le e^2\,\avdeg\,e^{-(\gamma-1)\avdeg/(C^*+\varepsilon)}
=o(1),
\end{align*}
as $\gamma>1$ and $\avdeg\to\infty$. 
\end{proof}

\subsection{Proofs of the two threshold theorems}
We finally have all ingredients to prove our main threshold results for Kemeny--Braess edges and nonedges. 
\begin{proof}[Proof of \Cref{thm:edge-threshold}]

For part~(i), the upper bound on $\avdeg$ forces
$p=O(\log n/n)=o(1)$.  We may therefore work on the events in
\Cref{lem:sparse-uniform-estimates}. 
In particular,  $G$ is connected
whp.
Since $2C^*I(1/2)=1$ (and, thus,
$2(C^*-\varepsilon)I(1/2)<1$),
continuity of $I$ permits us to choose
$0< \alpha<1/2$ sufficiently close to $1/2$ such that
\[
  2(C^*-\varepsilon)I(\alpha)<1.
\]
The assumed upper bound on $\avdeg$ then implies
\[
  n\exp\bigl(-2\avdeg I(\alpha)\bigr)\ge n^\delta
\]
for some fixed $\delta>0$.  By
\Cref{lem:sparse-low-degree-reservoirs}(ii), whp there is
an edge $e=uv$ with
\[  d_G(u)=d_G(v)=\lfloor\alpha\avdeg\rfloor
\]
such that, for $H=G-e$,  the harmonic neighbour statistics of both endpoints $u,v$ satisfy
\[
h_H(u),h_H(v) = 1+o(1)
\]
and neither endpoint belongs to a triangle.

We now verify that $H$ is connected.
The deterministic one-edge
calculation from the proof of \Cref{prop:sparse-addition}, now read
backwards, taking $X=H$ and $X^+=G$ 
in \eqref{eq:edge-deletion-S-bound}, \eqref{eq:edge-deletion-Pi-bound}, and \eqref{eq:edge-deletion-B_X+-bound}, gives
\[
  \nucnorm{S_H-S_G}=O(\avdeg^{-1}),
  \qquad
  \nucnorm{\Pi_H-\Pi_G}=O(\avdeg^{-1}),
\]
and
\[
  \opnorm{B_H}
  \le\opnorm{B_G}+O(\avdeg^{-1})
  =O(\avdeg^{-1/2})<1.
\]
Thus, $H$ is connected (see Glossary~\ref{glo:op}).

The right-hand side of~\cref{eq:sparse-cubic-exact} with $X=H$ for the pair  $u,v$ vanishes, and the Kemeny-tameness hypothesis
of \Cref{cor:sparse-deletion} holds.  Consequently,
\[
  \avdeg^2\bigl(\K(G-e)-\K(G)\bigr)
  =
  \frac{2(2\alpha-1)}{\alpha^2}+o(1)<0.
\]

Summing up, $H=G-e$ is connected and $\K(G-e)-\K(G)<0$. 
Thus $e$ is a $\kappa$-Braess edge, completing the proof of  part~(i).

For part~(ii), first suppose that $p=o(1)$.  By
\Cref{lem:sparse-uniform-estimates}(ii) and \Cref{lem:sparse-barriers}(ii), there are fixed $c,C,\eta>0$ such that,
whp,
$c\avdeg\le d_G(v)\le C\avdeg$  for every vertex $v$ and 
  $d_G(u)+d_G(v)>(1+\eta)\avdeg$ for every edge $uv$.
Fix such an edge and let $H=G-uv$.
The same perturbation estimate as above gives
\[
  \opnorm{B_H}=O(\avdeg^{-1/2})<1,
\]
uniformly over all edges, and therefore  $H=G-uv$  is connected for every edge $uv$.

The remaining conditions for Kemeny-tameness also transfer uniformly from $G$ to $H$.
Codegrees and triangle counts do not increase, and the minimum degree in
$H$ is at least $(c/2)\avdeg$.  Moreover, if
$a=d_G(u)$ and $b=d_G(v)$, then the identity
\begin{equation}\label{eq:sparse-harmonic-transfer}
  h_H(u)
  =
  \frac{a}{a-1}h_G(u)
  -\frac\avdeg{(a-1)b}
\end{equation}
holds, with the analogous identity at $v$.  Thus
\Cref{lem:sparse-uniform-estimates}(iii) implies
$h_H(u),h_H(v)=1+o(1)$, uniformly over all edges.
Applying \Cref{cor:sparse-deletion} gives
\[
  \K(G-uv)-\K(G)
  =
  \frac{2(d_G(u)+d_G(v)-\avdeg)}
       {\avdeg d_G(u)d_G(v)}
  +o(\avdeg^{-2})>0.
\]
Indeed, the displayed main term is at least
$2\eta C^{-2}\avdeg^{-2}$, % 
 whereas
the remainder is $o(\avdeg^{-2})$.  This proves part~(ii) whenever
$p=o(1)$.

It remains to remove that restriction.  Let $\mathcal F_n^{\rm e}$ be
the event that $G$ is disconnected, or that $G$ is connected and some
edge $e$ makes  $G-e$ disconnected or satisfies
$\K(G-e)\le\K(G)$.  If
$\Pp(\mathcal F_n^{\rm e})\not\to0$, then there is a subsequence along which
these probabilities are bounded away from zero.  Passing to a further
subsequence, compactness gives $p\to q\in[0,1]$.  If $q=0$, then the sparse
argument just proved gives $\Pp(\mathcal F_n^{\rm e})\to0$.  If $q>0$,
then eventually $p\ge q/2$, and \Cref{thm:dense-exclusion}(ii) gives the
same conclusion.  Both cases are contradictions, and part~(ii) follows for
arbitrary $p=p(n)$.
\end{proof}

\begin{proof}[Proof of \Cref{thm:nonedge-threshold}]
Again, $G$ is connected whp.  For part~(i), the upper
bound on $\avdeg$ forces $p=o(1)$.  Choose $0< \alpha<1/2$ so close to
$1/2$ that
$(2C^*-\varepsilon)I(\alpha)<1$.
Then,
\[
  n\exp\bigl(-\avdeg I(\alpha)\bigr)\ge n^\delta
\]
for some fixed $\delta>0$.  By
\Cref{lem:sparse-low-degree-reservoirs}(i), there are two vertices
$u,v$ of degree $\lfloor\alpha\avdeg\rfloor$ that are nonadjacent, have
no common neighbour, satisfy $h_G(u),h_G(v)=1+o(1)$, and belong to no
triangle.  The cubic term \eqref{eq:sparse-cubic-exact}  with $X=G$ is therefore
identically zero, and the Kemeny-tameness hypothesis
holds.  Proposition~\ref{prop:sparse-addition} with $X=G$  gives
\[
  \avdeg^2\bigl(\K(G+uv)-\K(G)\bigr)
  =
  \frac{2(1-2\alpha)}{\alpha^2}+o(1)>0,
\]
{implying $\K(G+uv)-\K(G)>0$.}
Thus $uv$ is a $\kappa$-Braess nonedge, proving part~(i).

For part~(ii), first suppose that $p=o(1)$.  By
\Cref{lem:sparse-barriers}(i), there is an
$\eta>0$ such that whp
\[
  d_G(v)>(1/2+\eta)\avdeg
\]
for every vertex $v$.  The maximum degree is $O(\avdeg)$  by \Cref{lem:sparse-uniform-estimates}(ii), and
\Cref{lem:sparse-uniform-estimates}(iii) gives
$h_G(x)=1+o(1)$ uniformly.  The Kemeny-tameness hypothesis of
\Cref{prop:sparse-addition} therefore holds uniformly over the nonedges.
For any nonedge $uv$,
\[
  \frac{d_G(u)+d_G(v)}\avdeg>1+2\eta.
\]
It follows from \eqref{eq:sparse-addition}, uniformly over all nonedges,
that
\[
  \K(G+uv)-\K(G)<0.
\]
More explicitly, the main term in \eqref{eq:sparse-addition} is at most
$-c_\eta\avdeg^{-2}$ for a fixed $c_\eta>0$, while the uniform remainder
is $o(\avdeg^{-2})$.  This proves part~(ii) whenever $p=o(1)$.

Let $\mathcal F_n^{\rm ne}$ be the event that $G$ is disconnected, or
that $G$ is connected and some nonedge $uv$ satisfies
$\K(G+uv)\ge\K(G)$.  If
$\Pp(\mathcal F_n^{\rm ne})\not\to0$, then we pass to a subsequence on which
these probabilities are bounded away from zero and then to a further
subsequence on which $p\to q\in[0,1]$.  If $q=0$, then the sparse argument above
gives $\Pp(\mathcal F_n^{\rm ne})\to0$; if $q>0$, then eventually
$p\ge q/2$, and \Cref{thm:dense-exclusion}(i) gives the same conclusion.
This contradiction completes the proof for arbitrary $p=p(n)$.
\end{proof}

\section{Conclusion and Outlook}
Our results determine, to first order, the disappearance thresholds
for Kemeny--Braess edges and nonedges in the connected
Erd\H{o}s--R\'enyi regime. Both thresholds arise from the same
endpoint-degree criterion, with their different constants reflecting
the different numbers of candidate edges and nonedges. In particular,
dense random graphs exhibit {neither Kemeny--Braess edges nor nonedges whp}.
This contrasts with the abundance of paradoxical additions for the
spectral gap (\Cref{thm_eldan}) and with the behaviour predicted
for deletions by Chung's conjecture (confirmed by~\Cref{thm:eldan-edges} in a weaker form). These observations raise a
broader question: how does the choice of a spectral measure of
random-walk performance determine its response to adding or removing
connections?

A natural framework for this question is provided by the inverse
spectral moments
$F_q(G):=\sum_{i=2}^{n}(1-\lambda_i)^{-q}$, where $q>0$ and
$1=\lambda_1>\lambda_2\geq\cdots\geq\lambda_n$ are the
transition-matrix eigenvalues of a connected graph.
We have $F_1(G)=\K(G)$, while, for each fixed graph,
$F_q(G)^{1/q}\to(1-\lambda_2)^{-1}$ as $q\to\infty$.
Thus increasing $q$ gives progressively greater weight to the slowest
spectral modes and connects Kemeny's constant to the reciprocal
spectral gap. Taking the $q$th root preserves the sign of an edge
perturbation.

The moment estimates in our proof also apply to this family when
$q$ is fixed. The Taylor expansion of $(1-x)^{-q}$ has quadratic
coefficient $q(q+1)/2$, while the constant and linear contributions
to $F_q$ are unchanged by an edge modification:
the number of vertices is unchanged and
$\sum_{i=2}^{n}\lambda_i=-1$ for every loopless graph.
Together with our estimates for the cubic and higher moments, this
gives, for fixed $p\in(0,1)$ and $q>0$, whp and uniformly over
nonedges $uv$,
\[
  F_q(G+uv)-F_q(G)
  =
  -\frac{q(q+1)}{p^2n^2}+o(n^{-2}),
  \qquad G\sim G(n,p).
\]
The corresponding deletion estimate has the opposite sign.
The error bounds may depend on $p$ and $q$. Consequently, the
absence of paradoxical additions and deletions in the dense regime
extends to every fixed inverse spectral moment.

The situation can change when $q=q(n)$ grows. For any fixed graph,
an edge modification that strictly improves the spectral gap
decreases $F_q$ for all sufficiently large $q$, and the reverse
holds for a modification that strictly worsens the spectral gap.
It is therefore natural to ask at what growth rates of $q$
paradoxical additions and deletions emerge, and when they constitute
a positive proportion of the available modifications.

For deletions, this programme connects directly to Chung's
conjecture. Define
\[
  r_q(G):=
  \frac{
    |\{e\in E(G):
       G-e\text{ is connected and }F_q(G-e)<F_q(G)\}|
  }{|E(G)|}.
\]
Can one find a sequence $q=q(n)\to\infty$ and a constant $c>0$
such that $r_q(G(n,p))\geq c$ whp, with sufficient quantitative
control to transfer these inequalities to the spectral gap?
Such a result would provide a route to
Conjecture~\ref{conj_chung}.

Two difficulties distinguish this question from simply taking
$q\to\infty$. First, approximation of the reciprocal spectral gap
must be precise enough to determine the sign of a single-edge
perturbation. Writing $R(G):=(1-\lambda_2)^{-1}$, the bounds
$R(G)\leq F_q(G)^{1/q}\leq(n-1)^{1/q}R(G)$ control the values,
but their errors need not be smaller than the changes under
consideration. Second, the proportion of beneficial deletions must
be bounded away from zero with high probability.
The positive-probability conclusion of
\cite[Corollary~1.5]{eldan2017braess} does not establish this.
An important question is whether the aggregate quantities $F_q$
permit sufficient concentration estimates while $q$ grows towards
the regime dominated by extremal eigenvalues.

Our results identify rare low-degree endpoints as the mechanism
governing Kemeny's constant near its disappearance thresholds.
The spectral-gap phenomenon demonstrates the influence of
extremal spectral modes even when degrees are tightly concentrated.
Understanding how these mechanisms interact as both the density
and $q$ vary would connect the present threshold results with
Chung's conjecture and a broader study of spectral responses to
local graph modifications.

A finer question within this programme concerns the critical
windows around our two thresholds. When the endpoint-degree sum
differs from $\avdeg$ by only a bounded amount, the leading
perturbation can be of order $\avdeg^{-3}$, beyond the precision of
our current $o(\avdeg^{-2})$ error bound. More precise expansions
would clarify how neighbourhood fluctuations and higher spectral
moments modify the degree criterion at the boundary.

\newpage

\addcontentsline{toc}{section}{References}
\bibliographystyle{plainurl}
\bibliography{bibliography.bib}

\clearpage
\appendix

\section{Glossary}

\paragraph*{Erd\H{o}s--R\'enyi random graphs.}
The Erd\H{o}s--R\'enyi random graph $G(n,p)$ is the probability space of
graphs with vertex set $[n]=\{1,\ldots,n\}$ in which each of the $\binom{n}{2}$
possible edges is present independently with probability $p$
\cite{ErdosRenyi}. 
We refer to the edge probability $p$ as the \emph{density} of the model,
since it is the expected fraction of the $\binom{n}{2}$ possible edges
that are present; a \emph{fixed positive density} means that $p\in(0,1)$ is
held constant as $n\to\infty$.
The Erd\H{o}s--R\'enyi random graph is the basic reference model for a ``typical'' graph,
and its properties commonly exhibit coarse or sharp thresholds as $p=p(n)$ grows.
Connectivity is the classical example:  {$G(n,p)$ is connected whp
if $np-\log n\to+\infty$, and disconnected whp
if $np-\log n\to-\infty$; see~\cite{ErdosRenyi}.}

\paragraph*{The Euclidean norm.}
\label{sect:eucl}
{For a vector $\bold{x}=(x_v)_{v\in V}\in\mathbb{R}^{V}$, its Euclidean norm is
\[
    \lVert \bold{x}\rVert_2
    :=
    \sqrt{\bold{x}^\top \bold{x}}
    =
    \left(\sum_{v\in V}x_v^2\right)^{1/2}.
\]
In particular, every standard basis vector $e_v$ has Euclidean norm one.}

\paragraph*{The operator norm.}
\label{glo:op}
For a real matrix $A\in\R^{m\times n}$, the \emph{operator norm} (or
spectral norm) is
\[
  \opnorm{A}
  =\max_{\bold{x}\in\R^n,\;\bold{x}\neq0}\frac{\lVert A\bold{x}\rVert_2}{\lVert \bold{x}\rVert_2},
\]
the largest factor by which $A$ can stretch the Euclidean length of a
vector; it equals the largest singular value of $A$
\cite[\S5.6]{HornJohnson}. 
The operator norm is the matrix norm
induced by the Euclidean norm. 
For the symmetric matrices used in this paper it
admits the spectral description $\opnorm{A}=\max_i|\lambda_i(A)|$
\cite[\S5.6]{HornJohnson}.  Being an induced norm, it is submultiplicative,
$\opnorm{AB}\le\opnorm{A}\,\opnorm{B}$ \cite[\S5.6]{HornJohnson}.  

The operator norm provides a convenient way to detect whether a graph is
connected and non-bipartite. More precisely, assume that $G$ has no
isolated vertices and let
\[
  S_G=D_G^{-1/2}A_GD_G^{-1/2},
  \qquad
  \Pi_G=\phi_G\phi_G^\top,
  \qquad
  B_G=S_G-\Pi_G.
\]
The multiplicity of the eigenvalue $1$ of $S_G$, equivalently the
multiplicity of the eigenvalue $0$ of the normalised Laplacian
$\mathcal L_G=I-S_G$, equals the number of connected components of $G$;
see~\cite[Lemma~1.7(iv) and (vi)]{Chung97}. Since $\Pi_G$ is the
orthogonal projection onto the particular $1$-eigenvector $\phi_G$, if
$G$ is disconnected, then there is another $1$-eigenvector of $S_G$
orthogonal to $\phi_G$. This vector is also a $1$-eigenvector of
$B_G$, and hence $\opnorm{B_G}=1$. If $G$ is connected, then $G$ is bipartite if and only if $2$ is an
eigenvalue of $\mathcal L_G$, equivalently if and only if $-1$ is an
eigenvalue of $S_G$; see~\cite[Lemma~1.7(v) and Lemma~1.8]{Chung97}. A
corresponding eigenvector is orthogonal to $\phi_G$, so it is also a
$-1$-eigenvector of $B_G$, and again $\opnorm{B_G}=1$.
Conversely, if $G$ is connected and non-bipartite, then $0$ is a simple
eigenvalue of $\mathcal L_G$, while all its other eigenvalues lie strictly
between $0$ and $2$; see again
\cite[Lemma~1.7(iv)--(vi)]{Chung97}. Therefore every eigenvalue of $S_G$
other than its simple eigenvalue $1$ has modulus strictly less than $1$.
Since $B_G\phi_G=0$ and $B_G$ agrees with $S_G$ on the orthogonal
complement of $\phi_G$, it follows that $\opnorm{B_G}<1$. Consequently,
\begin{align}
  \opnorm{B_G}<1
  \quad\Longleftrightarrow\quad
  G\text{ is connected and non-bipartite}.
\label{eq:non-bip}
\end{align}

\paragraph*{The nuclear norm.}
\label{sect:nuclear}
For a real matrix $A\in\R^{m\times n}$, the \emph{nuclear norm} (or trace
norm, or Schatten $1$-norm) is
$\nucnorm{A}=\sum_i\sigma_i(A)$,
the sum of the singular values of $A$ \cite[\S IV.2]{Bhatia}; for
symmetric $A$ this equals $\sum_i|\lambda_i(A)|$.  It dominates the
operator norm, $\opnorm{A}\le\nucnorm{A}$, and since a matrix of rank
$k$ has at most $k$ nonzero singular values, it satisfies the rank--norm
inequality $\nucnorm{A}\le k\,\opnorm{A}$ \cite[\S IV.2]{Bhatia}.  
The outer product
$\bold{x}\bold{y}^\top$ has rank at most one, and 
thus we have the identity
\begin{align}
    \lVert \bold{x}\bold{y}^\top\rVert_*
    =
    \lVert \bold{x}\rVert_2\lVert \bold{y}\rVert_2.
    \label{eq:nucl-eucl}
\end{align}
Three 
further properties are used in this paper: 
the \emph{triangle inequality} 
\begin{align}
\nucnorm{A+B}\leq\nucnorm{A}+\nucnorm{B}
\label{eq:nuclear-triangle}
\end{align}
for all matrices $A,B$ of the same dimensions~\cite[Theorem~IV.2.1]{Bhatia}, 
the \emph{trace duality} \cite[\S IV.2]{Bhatia} (an instance of the matrix H\"older inequalities for Schatten norms) 
\begin{align}
|\tr(XY)|\le\opnorm{X}\,\nucnorm{Y} \label{eq:trace-duality}
\end{align}
 and the
inequality \begin{align}
\label{eqn_1608_2109}
\nucnorm{XYZ}\le\opnorm{X}\,\nucnorm{Y}\,\opnorm{Z},
\end{align}
valid
for every unitarily invariant norm \cite[Proposition~IV.2.4]{Bhatia}.  In
particular $\nucnorm{\bold{x}\bold{y}^{\mathsf T}+\bold{y}\bold{x}^{\mathsf T}}
\le2\lVert \bold{x}\rVert_2\lVert \bold{y}\rVert_2$, the form in which the new edge
enters the perturbation estimates.

\paragraph*{Asymptotic notation.}
All asymptotic notation in this work refers to the limit $n\to\infty$.
For real-valued functions $f=f(n)$ and $g=g(n)$, with $g(n)>0$
for all sufficiently large $n$, we write $f=O(g)$ if there exists
a constant $C>0$ such that $|f(n)|\le Cg(n)$ for all sufficiently
large $n$, and $f=o(g)$ if $f(n)/g(n)\to0$.
Moreover, we write $f=\Omega(g)$ if there exists
a constant $c>0$ such that $|f(n)|\ge cg(n)$ for all sufficiently
large $n$, and $f=\Theta(g)$ if both $f=O(g)$ and $f=\Omega(g)$.
Unless stated otherwise, implicit constants may depend on fixed
parameters but not on $n$.

\section{Omitted proof}
\label{appendix_omitted_proofs}%

{
In this section, we prove~\Cref{lem:event}.
We shall make use of the following matrix Bernstein theorem. 
\begin{theorem}[\protect{\cite[Theorem~1.4]{Tropp}}]
\label{thm_matrix_bernstein}
Consider a finite sequence $\{X_k\}$ of independent, random, self-adjoint
matrices with dimension $d$. Assume that each random matrix satisfies
\[
\mathbb{E}X_k=0
\qquad\text{and}\qquad
\lambda_{\max}(X_k)\leq R
\quad\text{almost surely}.
\]
Then, for all $t\geq 0$,
\[
\mathbb{P}\left\{
  \lambda_{\max}\left(\sum_k X_k\right)\geq t
\right\}
\leq
d\exp\left(
  \frac{-t^2/2}{\sigma^2+Rt/3}
\right),
\]
where
\[
\sigma^2
:=
\left\|
  \sum_k \mathbb{E}\left(X_k^2\right)
\right\|_{\operatorname{op}}.
\]
\end{theorem}}

\begin{proof}[Proof of~\Cref{lem:event}]
We use the following standard multiplicative Chernoff bound: if $X$ is
binomial with mean $m$, then, for $0<\tau\leq 1$,
\begin{equation}
  {\Pp}\bigl(|X-m|>\tau m\bigr)
  \leq 2\exp\left(-\frac{\tau^2m}{3}\right).
  \label{eq:Chernoff-dense-event}
\end{equation}
Set $a:=\frac{6}{p_0}$.
Since $\varepsilon\coloneqq\varepsilon_n\to0$ as $n\to\infty$, we may assume throughout that $n$ is
sufficiently large that $a\varepsilon\leq 1/2$.  
For each vertex $v$, consider the events
\[
  \mathcal D_v
  :=
  \left\{
    \left|\frac{d_v}{\avdeg}-1\right|
    \leq a\varepsilon
  \right\},
  \qquad
  \mathcal D:=\bigcap_v\mathcal D_v.
\]
Since  $d_v\sim \mathrm{Bin}(n-1,p)$ with mean
$\avdeg=p(n-1)$, \eqref{eq:Chernoff-dense-event} and the union bound give
\[
\begin{aligned}
  \Pp(\mathcal D^{\,c})
  &\leq
  2n\exp\left(
    -\frac{a^2\varepsilon^2\avdeg}{3}
  \right).
\end{aligned}
\]
For $n\geq2$,
\[
  \varepsilon^2\avdeg
  =
  \frac{\log n}{n}p(n-1)
  \geq \frac{p_0}{2}\log n.
\]
Moreover,
\[
  \frac{a^2p_0}{6}
  =\frac{6}{p_0}\geq6.
\]
Consequently,
\begin{equation}
\label{eqn_1552_1008}
  \sup_{p\in[p_0,1]}\Pp(\mathcal D^{\,c})
  \leq 2n^{-5}.
\end{equation}
On $\mathcal D$, condition~\textnormal{(i)} in the definition of
$\mathcal E_{n,p}$ holds with constant $a$.

We now look at condition (ii). For distinct vertices $u,v$, put
\[
  c_{uv}:=|N(u)\cap N(v)|.
\]
There are $n-2$ possible common neighbours, and each is adjacent to
both $u$ and $v$ with probability $p^2$.  Hence
$
  c_{uv}\sim\operatorname{Bin}(n-2,p^2),
$
and its mean is
$
  m_2:=(n-2)p^2.
$
Consider the event
\[
  \mathcal Q
  :=
  \bigcap_{u\neq v}
  \left\{
    \left|\frac{c_{uv}}{m_2}-1\right|
    \leq a\varepsilon
  \right\}.
\]
By \eqref{eq:Chernoff-dense-event} and a union bound over fewer than
$n^2$ ordered pairs,
\[
  \Pp(\mathcal Q^{\,c})
  \leq
  2n^2\exp\left(
    -\frac{a^2\varepsilon^2m_2}{3}
  \right).
\]
For $n\geq4$,
\[
  \varepsilon^2m_2
  =
  \frac{\log n}{n}(n-2)p^2
  \geq\frac{p_0^2}{2}\log n.
\]
Since $a^2p_0^2/6=6$, it follows that
\begin{equation}
  \sup_{p\in[p_0,1]}\Pp(\mathcal Q^{\,c})
  \leq 2n^{-4}.
  \label{eq:codegree-event-probability}
\end{equation}

Notice that $\mathcal Q$ automatically implies connectivity of $G$.  Indeed, for all
sufficiently large $n$, on $\mathcal Q$ we have, for every $u\neq v$,
\[
  |N(u)\cap N(v)|
  \geq
  (1-a\varepsilon)(n-2)p^2
  \geq
  \frac{p_0^2(n-2)}{2}
  >0.
\]
Thus every two distinct vertices have a common neighbour.  In
particular, $G$ is connected.

For a vertex $v$, define
\[
  q_v:=p\binom{d_v}{2},
  \qquad
  \tau_n:=\binom{n-1}{2}p^3.
\]
Conditional on the neighbourhood $N(v)$, the edges with both endpoints
in $N(v)$ remain mutually independent and are present with probability
$p$.  Therefore, recalling that $t_v$ denotes the number of triangles in $G$ containing $v$, we have
\[
  t_v\mid N(v)
  \sim
  \operatorname{Bin}\left(\binom{d_v}{2},p\right),
\]
whose mean is $q_v$.
On $\mathcal D_v$,
\[
  d_v\geq(1-a\varepsilon)\avdeg
  \geq\frac{p_0(n-1)}{2}.
\]
Thus, for all sufficiently large $n$,
\[
  q_v
  =
  p\binom{d_v}{2}
  \geq
  \frac{p_0^3}{64}n^2.
\]
Since $\mathcal D_v$ and $q_v$ are determined by $N(v)$, applying
\eqref{eq:Chernoff-dense-event} gives
\begin{align*}
 &\Pp\left(
   \mathcal D_v
   \cap
   \bigl\{|t_v-q_v|>\varepsilon q_v\bigr\}
 \right)                                                    \\
 &\quad=
 \mathbb E\left[
   \mathbf 1_{\mathcal D_v}
   \Pp\left(
     |t_v-q_v|>\varepsilon q_v
     \,\middle|\,N(v)
   \right)
 \right]                                                     \\
 &\quad\leq
 2\exp\left(-\frac{\varepsilon^2p_0^3n^2}{192}\right)
 =
 2\exp\left(-\frac{p_0^3}{192}n\log n\right).
\end{align*}
Define
\[
  \mathcal R
  :=
  \bigcap_v
  \left(
    \mathcal D_v^{\,c}
    \cup
    \bigl\{|t_v-q_v|\leq\varepsilon q_v\bigr\}
  \right).
\]
The preceding estimate and a union bound yield
\begin{equation}
\label{eqn_1553_1008_b}
  \sup_{p\in[p_0,1]}\Pp(\mathcal R^{\,c})
  \leq
  2n\exp\left(-\frac{p_0^3}{192}n\log n\right).
\end{equation}
It remains to compare the conditional mean $q_v$ with the deterministic
quantity $\tau_n$.  On $\mathcal D_v$, write
\[
  d_v=\avdeg(1+\eta_v)
      =p(n-1)(1+\eta_v),
  \qquad
  |\eta_v|\leq a\varepsilon.
\]
A direct calculation gives
\begin{align}
  \frac{q_v}{\tau_n}
  &=
  \frac{d_v(d_v-1)}{p^2(n-1)(n-2)}
  =
  \frac{n-1}{n-2}(1+\eta_v)^2
  -
  \frac{1+\eta_v}{p(n-2)}.
  \label{eq:conditional-triangle-mean}
\end{align}
Consequently, using the triangle inequality,
\begin{align*}
  \left|\frac{q_v}{\tau_n}-1\right|
  &\leq
  \frac{1}{n-2}
  +
  \frac{n-1}{n-2}
  \bigl(2|\eta_v|+\eta_v^2\bigr)
  +
  \frac{1+|\eta_v|}{p(n-2)}.
\end{align*}
For all sufficiently large $n$,
\[
  \frac1{n-2}\leq\varepsilon,
  \qquad
  \frac1{p_0(n-2)}\leq\varepsilon,\qquad\mbox{and}
  \qquad
  a\varepsilon\leq\frac12.
\]
Using also $(n-1)/(n-2)\leq2$, we obtain
\[
  \left|\frac{q_v}{\tau_n}-1\right|
  \leq
  (5a+3)\varepsilon.
\]
Put
$b:=5a+3$.
After increasing $n$ further so that $b\varepsilon\leq1$, we have
$q_v/\tau_n\leq2$.  Therefore, on
$\mathcal D\cap\mathcal R$,
\begin{align*}
  \left|\frac{t_v}{\tau_n}-1\right|
  &\leq
  \frac{|t_v-q_v|}{\tau_n}
  +
  \left|\frac{q_v}{\tau_n}-1\right|                         \\
  &\leq
  \varepsilon\frac{q_v}{\tau_n}
  +b\varepsilon
  \leq
  (b+2)\varepsilon
\end{align*}
for every $v$ (where the second inequality uses that, on $\mathcal D\cap\mathcal R$, it holds that $|t_v-q_v|\leq\varepsilon q_v$ for each $v$). This proves condition~\textnormal{(iii)} with constant
$b+2$.

We now turn to (iv). First, we establish a concentration estimate for the adjacency matrix $A$ of $G$.
For $1\leq i<j\leq n$, let $\xi_{ij}$ be the indicator of the edge
$ij$, and recall that $e_i$ denotes the $i$-th standard basis vector. 
Set also
\[
  X_{ij}
  :=
  (\xi_{ij}-p)
  \bigl(
    e_ie_j^\top
    +
    e_je_i^\top
  \bigr).
\]
Then the $X_{ij}$ are independent, symmetric, mean-zero random matrices,
and
\[
  \sum_{i<j}X_{ij}=A-p(J-I),
\]
where $J$ is the all-one matrix.
Letting $M_{ij}=e_ie_j^\top+e_je_i^\top$ and noting that $\|M_{ij}\|_{\mathrm{op}}=1$, we find
\begin{align*}
    \|X_{ij}\|_{\mathrm{op}}
=
|\xi_{ij}-p|\,\|M_{ij}\|_{\mathrm{op}}
=
|\xi_{ij}-p|
\leq 1.
\end{align*}
Moreover,
\[
\begin{aligned}
  \sum_{i<j}\mathbb E(X_{ij}^2)
  &=
  p(1-p)
  \sum_{i<j}
  \bigl(
    e_ie_i^\top
    +
    e_je_j^\top
  \bigr)                                                     =
  p(1-p)(n-1)I.
\end{aligned}
\]
Applying~\Cref{thm_matrix_bernstein} to the $X_{ij}$ and to the $-X_{ij}$, and
then taking a union bound, gives, for every $t\geq0$,
\begin{equation}
  \Pp\left(
    \opnorm{A-p(J-I)}\geq t
  \right)
  \leq
  2n\exp\left(
    -\frac{t^2}
    {2p(1-p)(n-1)+2t/3}
  \right).
  \label{eq:matrix-Bernstein-adjacency}
\end{equation}
Let
\[
  \mathcal M
  :=
  \left\{
    \opnorm{A-p(J-I)}
    \leq 4\sqrt{n\log n}
  \right\}.
\]
For all sufficiently large $n$,
$4\sqrt{n\log n}\leq n$.  Since
$p(1-p)(n-1)\leq n$,~\eqref{eq:matrix-Bernstein-adjacency} gives
\begin{equation}
  \sup_{p\in[p_0,1]}\Pp(\mathcal M^{\,c})
  \leq
  2n\exp\left(-\frac{16}{3}\log n\right)
  =
  2n^{-13/3}.
  \label{eq:matrix-event-probability}
\end{equation}

We now work on
$\mathcal D\cap\mathcal Q\cap\mathcal M$.  As we have seen, the event $\mathcal Q$
implies that $G$ is connected, so the vector $\phi$ and the matrices $S,\Pi,B$ (see~\Cref{sec_spectral_expansion}) are
well-defined. As before, write $d_v=\avdeg(1+\eta_v)$ with
  $|\eta_v|\leq a\varepsilon$,
and define the diagonal matrix
\[
  F:=\avdeg^{1/2}D^{-1/2}-I.
\]
Thus
\[
  D^{-1/2}=\avdeg^{-1/2}(I+F).
\]
For $|x|\leq1/2$, the mean-value theorem gives
$\left|(1+x)^{-1/2}-1\right|\leq2|x|$.
It follows that $\opnorm{F}\leq \max_v 2|\eta_v| \leq2a\varepsilon$.
Furthermore,
\[
  \opnorm{A}
  \leq \max_v d_v
  \leq \avdeg(1+a\varepsilon)
  \leq \frac32\avdeg.
\]
Consequently,
\begin{align}
  \opnorm{S-\avdeg^{-1}{A}}
  &=
  \avdeg^{-1} 
  \opnorm{FA+AF+FAF}                                    \leq
 \avdeg^{-1} \opnorm{A} 
  \left(
    2\opnorm{F}+\opnorm{F}^2
  \right)                                                
  \leq 9a\varepsilon.
  \label{eq:S-adjacency-comparison}
\end{align}

Next, put
\[
  \overline\eta:=\frac1n\sum_v\eta_v,
  \qquad
  \psi:=\frac1{\sqrt n}\one.
\]
Since $2|E|
  =
  \sum_vd_v
  =
  n\avdeg(1+\overline\eta)$,
{the $v$-th entry of the vector $\phi$ is}
\[
  \phi_v
  =
  \frac1{\sqrt n}
  \sqrt{\frac{1+\eta_v}{1+\overline\eta}}.
\]
Both $|\eta_v|$ and $|\overline\eta|$ are at most
$a\varepsilon$, and hence
\begin{align*}
  \left|
    \sqrt{\frac{1+\eta_v}{1+\overline\eta}}-1
  \right|
  &=
  \frac{|\eta_v-\overline\eta|}
  {
    \sqrt{1+\overline\eta}
    \bigl(
      \sqrt{1+\eta_v}+\sqrt{1+\overline\eta}
    \bigr)
  }                                                          \leq
  |\eta_v-\overline\eta|
  \leq2a\varepsilon.
\end{align*}
It follows that
\[
  \lVert\phi-\psi\rVert_2\leq2a\varepsilon.
\]
Because $\phi$ and $\psi$ are unit vectors,
\begin{align}
  \opnorm{\Pi-\frac{J}{n}}
  &=
  \opnorm{\phi\phi^\top-\psi\psi^\top}                  
  \leq
  2\lVert\phi-\psi\rVert_2
  \leq4a\varepsilon.
  \label{eq:Pi-comparison}
\end{align}

Finally, on $\mathcal M$,
\begin{align*}
  \left\|{\frac{A}{\avdeg}-\frac{J}{n}}\right\|_{\operatorname{op}}
  &\leq
  \frac{\opnorm{A-p(J-I)}}{p(n-1)}
  +
  \left\|{\frac{J-I}{n-1}-\frac{J}{n}}\right\|_{\operatorname{op}}                       \\
  &\leq
  \frac{4\sqrt{n\log n}}{p_0(n-1)}
  +
  \frac1{n-1}.
\end{align*}
Here
\[
  \frac{J-I}{n-1}-\frac{J}{n}
  =
  \frac{J/n-I}{n-1},
\]
whose operator norm is $1/(n-1)$.  Moreover,
\[
  \frac{\sqrt{n\log n}}{n-1}
  =
  \frac{n}{n-1}\varepsilon
  \leq2\varepsilon,
\]
and $1/(n-1)\leq\varepsilon$ for all sufficiently large $n$.
Therefore
\begin{equation}
    \left\|{\frac{A}{\avdeg}-\frac{J}{n}}\right\|_{\operatorname{op}}
  \leq
  \left(\frac{8}{p_0}+1\right)\varepsilon.
  \label{eq:centered-adjacency-comparison}
\end{equation}
Combining
\eqref{eq:S-adjacency-comparison},
\eqref{eq:Pi-comparison}, and
\eqref{eq:centered-adjacency-comparison}, we obtain
\begin{align*}
  \opnorm{B}
  =
  \opnorm{S-\Pi}
  &\leq
  \opnorm{S-\frac{A}{\avdeg}}
  +
  \left\|{\frac{A}{\avdeg}-\frac{J}{n}}\right\|_{\operatorname{op}}
  +
  \left\|{\frac{J}{n}-\Pi}\right\|_{\operatorname{op}}                                  \\
  &\leq
  \left(
    13a+\frac{8}{p_0}+1
  \right)\varepsilon.
\end{align*}
{Set $C:=\max\{a,b+2,13a+\frac{8}{p_0}+1\}$, and note that $C$ depends only on $p_0$.}
Define the auxiliary event
\[
  \mathcal F_{n,p}
  :=
  \mathcal D\cap\mathcal Q\cap\mathcal R\cap\mathcal M.
\]
The preceding arguments show that, for all sufficiently large $n$,
$\mathcal F_{n,p}\subseteq\mathcal E_{n,p}$.
Moreover, by
\eqref{eqn_1552_1008},
\eqref{eq:codegree-event-probability},
\eqref{eqn_1553_1008_b}, and
\eqref{eq:matrix-event-probability},%
\begin{align*}
  \sup_{p\in[p_0,1]}
  \Pp(\mathcal E_{n,p}^{\,c})
  &\leq
  \sup_{p\in[p_0,1]}
  \Pp(\mathcal F_{n,p}^{\,c})                                \\
  &\leq
  2n^{-5}
  +2n^{-4}
  +2n\exp\left(-\frac{p_0^3}{192}n\log n\right)
  +2n^{-13/3}.
\end{align*}
The right-hand side tends to zero as $n\to\infty$.  Hence
\[
\lim_{n\to \infty}\inf_{p\in[p_0,1]}\Pp(\mathcal E_{n,p})=1,
\]
as required.
\end{proof}

\end{document}